\documentclass[11pt]{article}
\usepackage{stmaryrd}
\usepackage{amsfonts}
\usepackage{bbm}

\usepackage{anyfontsize}

\usepackage[all,cmtip]{xy}

\usepackage{amscd}
\usepackage{mathrsfs}
\usepackage{latexsym,amssymb,amsmath,amscd,amscd,amsthm,amsxtra}
\usepackage[dvips]{graphicx}
\usepackage[utf8]{inputenc}
\usepackage[T1]{fontenc}
\usepackage{enumerate}
\usepackage{lmodern}
\usepackage{amssymb}
\usepackage[all]{xy}
\usepackage{nicefrac,mathtools,enumitem}
\usepackage{microtype}
\usepackage{indentfirst}
\usepackage{cite}
\usepackage{kpfonts}
\usepackage{geometry}
\usepackage{ifpdf}
\ifpdf
  \usepackage[colorlinks=true,linkcolor=blue,citecolor=blue, final,backref=page,hyperindex]{hyperref}
\else
  \usepackage[colorlinks,final,backref=page,hyperindex,hypertex]{hyperref}
\fi

\allowdisplaybreaks

\makeatletter
\newcommand{\subjclass}[2][2020]{%
  \let\@oldtitle\@title%
  \gdef\@title{\@oldtitle\footnotetext{#1 \emph{Mathematics subject classification}: #2}}%
}
\newcommand{\keywords}[1]{%
  \let\@@oldtitle\@title%
  \gdef\@title{\@@oldtitle\footnotetext{\emph{Keywords}: #1}}%
}
\makeatother

\def\c{\cdot}

\def\sw{\swarrow}
\def\nw{\nwarrow}
\def\ne{\nearrow}
\def\se{\searrow}

\newtheorem{thm}{Theorem}[section]
\newtheorem{prop}{Proposition}[section]
\newtheorem{lem}{Lemma}[section]
\newtheorem{rem}{Remark}[section]

\newtheorem{cor}{Corollary}[section]
\newtheorem{defn}{Definition}[section]

\newtheorem{ex}{Exemple}[section]

\numberwithin{equation}{section}
\newcommand{\las}{(\!(}
\newcommand{\ras}{)\!)}
\date{}
\newcommand{\Addresses}{{% additional braces for segregating \footnotesize
  {\centering
{\small $^1$ University of Sfax, Faculty of Sciences Sfax,  BP
1171, 3038 Sfax, Tunisia} \\
{\small $^2$ University of Gafsa, Faculty of Sciences Gafsa, 2112 Gafsa, Tunisia}\\
{\small $^3$ M\"{a}lardalen University,
Division of Mathematics and Physics,}\\
{\small \hspace{1.5 cm}
School of Education, Culture and Communication, }\\
{\small \hspace{1.5 cm}
Box 883, 72123 V\"{a}steras, Sweden.}\\
}}}

\begin{document}

\title{Dendrification of Hom-Malcev algebras}
%\author{ F. Harrathi, S. Mabrouk, O. Ncib,
%S. Silvestrov}
\author{F. Harrathi $^{1}$
\footnote{E-mail: harrathifattoum285@gmail.com}\ , 
 S. Mabrouk $^{2}$
\footnote{E-mail: mabrouksami00@yahoo.fr (Corresponding author)}\ ,
O. Ncib $^{2}$\footnote{E-mail: othmenncib@yahoo.fr}\ , 
S. Silvestrov $^{3}$\footnote{ E-mail:sergei.silvestrov@mdu.se} \\[0.1cm]
}

\maketitle

\Addresses

\abstract{The main goal of this work is to introduce the notion of Hom-M-dendriform algebras which
are the dendriform version of Hom-Malcev algebras. In fact they are the algebraic structures behind
the $\mathcal{O}$-operator of Hom-pre-Malcev algebras.
They also fit into a bigger framework as Hom-Malcev algebraic analogues of Hom-L-dendriform algebras. Furthermore, we show a connections between Hom-M-Dendriform algebras and Hom-alternative quadri-algebras.

\vspace{0.5 cm}
}

%\noindent\textbf{Keywords:}{Hom-pre-Malcev algebra, Hom-M-dendriform algebra, Rota-Baxter opertor} \\
%\noindent{\textbf{MSC2020:}}{ 17D10, 17A01, 17A30, 17B10.}
%\end{abstract}

\tableofcontents

\numberwithin{equation}{section}
\section{Introduction}
Rota-Baxter operators were introduced by G. Baxter \cite{baxter} in 1960 in the study of fluctuation theory in Probability. Such operators
$\mathcal{R}: A\to A$ are defined on an algebra $A$ by the identity
\begin{align}
\mathcal{R}(x)\mathcal{R}(y) = \mathcal{R}(\mathcal{R}(x)y + x\mathcal{R}(y) + \lambda xy),
\label{rotabaxterop}
\end{align}
for all $x, y\in A$, where $\lambda$ is a scalar, called the weight of $\mathcal{R}$. These operators were then further investigated, by G.-C. Rota \cite{Rota}, Atkinson \cite{Atkinson}, Cartier \cite{Cartier} and others. Since the late 1990s, the study of Rota-Baxter operators has made great progress in many areas, both in theory and in applications \cite{Aguiar, EbrahimiFardGuo05,EbrahimiFardGuo08,EbrahimiFardGuo09, Guo0,Guo1,Guo2}. Rota-Baxter
algebras allow us to define dendriform dialgebra ($\lambda = 0$, see \cite{Aguiar00}
) and trialgebra
structures ($\lambda \neq 0$, see \cite{EbrahimiFard02}) in associative algebras.

In this article, we deal with a specific type of algebraic structures, called Hom-algebraic structures. The defining identities of such an algebraic structure are twisted by endomorphisms. These structures first appeared in the form of Hom-Lie algebras and more general quasi Hom-Lie algebras \cite{HartwigLarSil:defLiesigmaderiv}, in the context of deformation of Witt and Virasoro algebras using general twisted derivations called $\sigma$-derivations.
Hom-algebras have been studied extensively in the works of J. Hartwig, D.
Larsson, A. Makhlouf, S. Silvestrov, D. Yau and other authors \cite{AmmarEjbehiMakhlouf:homdeformation,HartwigLarSil:defLiesigmaderiv,ms:homstructure,MakhSilv:HomDeform,Yau:HomHom}. Other types of
algebras, such as associative, Leibniz, Poisson, Hopf, Malcev, etc., twisted by homomorphisms have also
been studied in the last few years.

Recently, there have been several interesting developments of Hom-Lie algebras in mathematics
and mathematical physics, including Hom-Lie bialgebras \cite{CHENWANGZHANG1,CHENWANGZHANG2}, quadratic Hom-Lie algebras
\cite{BenMakh:Hombiliform}, involutive Hom-semigroups \cite{zhengguo}, deformed vector fields and differential calculus \cite{LarssonSilvJA2005:QuasiHomLieCentExt2cocyid}, representations
\cite{Sheng:homrep}, cohomology and homology theory \cite{AmmarEjbehiMakhlouf:homdeformation,Yau:HomHom},
Hom-Yang-Baxter equations \cite{ChenZhang1,Sheng:homrep,Yau:HomYangBaHomLiequasitribial}, Hom-Lie $2$-algebras  \cite{sheng,Songtang},
$(m, n)$-Hom-Lie algebras
\cite{MaZheng}, Hom-left-symmetric algebras \cite{ms:homstructure} and enveloping algebras \cite{Yau:HomEnv}.

%The twisted version of another type algebras, called dendriform algebras,

The notion of dendriform algebras
was introduced in 1995 by J.-L. Loday \cite{Loday01} in his study of algebraic $K$-theory. Dendriform algebras
are algebras with two operations, which dichotomize the notion of associative algebras. Later the notion of tridendriform algebra were
introduced by Loday and Ronco in their study of polytopes and Koszul duality (see \cite{LodayRonco04}).
In order to determine the algebraic structures behind a pair of commuting Rota-Baxter operators (on an associative algebra), Aguiar and Loday introduced the notion of quadri-algebra in \cite{AL}.

Another important class of nonassociative algebras is the Malcev
algebras which were introduced in 1955 by A.I.Malcev \cite{Malcev55:anltcloops}, generalizing Lie algebras. They play an important role in Physics and the geometry of smooth loops. Just as the tangent
algebra of a Lie group is a Lie algebra, the tangent algebra of a locally analytic Moufang
loop is a Malcev algebra \cite{kerdman,kuzmin68,kuzmin71:conectMalcevanMoufangloops,Malcev55:anltcloops,nagy,sabinin}, see also \cite{gt,myung,okubo} for discussions about connections with physics.

The Hom-type generalization of Malcev algebras, called Hom-Malcev
algebras, is defined in \cite{Yau}, where connections between Hom-alternative algebras and Hom-Malcev algebras are given.
\begin{defn}[\cite{Yau}]
A Hom-Malcev algebra  is a vector space $A$ together with a skew-symmetric linear map $[-, -] :\otimes^{2}A\to A$  and a l
inear map $\alpha: A \to A$ such that, for all $x, y, z \in A$,
\begin{equation}
\label{eq:Hom-Malcev:Jacobiannotation}
J_{\alpha}(\alpha (x), \alpha (y), [x,z]) = [J_{\alpha}(x, y, z), \alpha^{2}(x)],
\end{equation}
where $J_{\alpha}(x, y, z)=[[x,y],\alpha (z)]+[[y,z],\alpha (x)]+[[z,x],\alpha (y)]$ is the Hom-Jacobian of $x, y, z$ \textup{(}$\alpha$-Jacobian\textup{)}.
The Hom-Malcev identity \eqref{eq:Hom-Malcev:Jacobiannotation} is equivalent to
\begin{equation}
\label{Hom-Malcev}
\begin{array}{ll}
[\alpha ([x, z]), \alpha ([y, t])] &= [[[x, y], \alpha (z)], \alpha^{2}(t)] + [[[y, z], \alpha (t)], \alpha^{2} (x)]\\
& \quad + [[[z, t], \alpha (x)], \alpha^{2} (y)] + [[[t, x], \alpha (y)], \alpha^{2} (z)].
\end{array}
\end{equation}
\end{defn}

 In \cite{Madariaga}, in order to find a dendriform algebra whose anti-commutator is a pre-Malcev
algebra, M-dendriform algebras were introduced, an $\mathcal{O}$-operator (specially a Rota-Baxter operator of weight zero) on a pre-Malcev algebra or two commuting Rota-Baxter operators on a Malcev algebra were shown to give a M-dendriform algebra, and also the relationships between M-dendriform algebras and Loday algebras especially quadri-algebras was established.

% Rota-Baxter algebras first appeared in 1960 in the works by Baxter on probability.
%They became popular in different areas of mathematics and physics and recently they have been used in areas such as quantum field theory,
%Yang-Baxter equations, shuffle products, operads, Hopf algebras, combinatorics or number theory (see \cite{EbrahimiFardGuo08} and the references therein).

 Our main objective is to introduce the notion of Hom-M-dendriform algebras which are the twisted analog of M-dendriform algebras, where the
M-dendriform algebra identity is twisted by a self linear map, called the structure map and study, through Rota-Baxter operators and $\mathcal{O}$-operators, the relationship between  Hom-pre-Malcev algebras and Hom-M-dendriform algebras.

Hom-M-dendriform algebra gives rise to a Hom-pre-Malcev algebra and a Hom-Malcev algebra, in the same way as  Hom-L-dendriform algebra, gives rise to Hom-pre-Lie algebra and Hom-Lie algebra (see \cite{ChtiouiMabroukMakhlouf1,ChtiouiMabroukMakhlouf2,MakhloufHomdemdoformRotaBaxterHomalg2011} for more details).

We study in the second section the representation of Hom-Malcev algebras and recall the definitions and
some properties of Hom-pre-Malcev algebras. The notion of a representation of a Hom-pre-Malcev algebra was introduced in \cite{Fattoum}. However, one need to add some complicated
conditions to obtain the dual representation, which is not convenient in applications. In
this paper, we give the natural formula of a dual representation, which is nontrivial. This is the content of Section $3$.
In Section $4$, we introduce the
 notion of Hom-M-dendriform algebra and then study some of their fundamental properties  in terms of the $\mathcal{O}$-operators of Hom-pre-Malcev algebras. Then we define a Hessian structure on a regular
Hom-pre-Malcev algebra $(A, \cdot,\alpha)$ to be a nondegenerate 2-cocycle $\mathcal{B} \in Sym^{2}A^{*}$, and show that there
is a relation between  Hom-M-dendriform algebras and Hessian structures on $(A, \cdot,\alpha)$. Their relationship with
Hom-Malcev algebras and Hom-alternative quadri-algebras are also described in Section $5$.

%%%%%%%%%%%%%%%%%%%%%%%%%%%%%%%%%%%%%%%%%%%%%%%%%%%%%%%%%%%
\section{Basics on Hom-Malcev and Hom-pre-Malcev algebras}\label{sec: prelimenaires}
%%%%%%%%%%%%%%%%%%%%%%%%%%%%%%%%%%%%%%%%%%%%%%%%%%%%%%%%%%%%%
In this section, we recall representations and dual representations of Hom-Malcev  algebras  and the definition of Hom-pre-Malcev algebras which are
useful throughout this paper.

In this paper, all vector spaces are over a field $\mathbb{K}$ of characteristic $0$.

A Hom-Malcev algebra  $(A,[-, -], \alpha)$ is said to be a \textbf{regular Hom-Malcev algebra } if $\alpha$ is invertible.  We denote by $End(A)$ the space of all linear map $f:A\to A$.

The notion of  representation of Hom-Malcev algebra was introduced in \cite{kuzmin68,Fattoum}.
\begin{defn}
Let $(A,[-, -], \alpha)$ be a Hom-Malcev algebra. A  representation  of $(A,[-, -], \alpha)$ on a vector space  $V$ with respect to $\beta\in End(V)$ is a linear map $\rho : A \longrightarrow End(V )$ satisfying, for any $x, y, z \in A,$
\begin{eqnarray} \label{rephommalcev}
\rho(\alpha(x))\beta &=&  \beta\rho(x),\label{rephommalcev1}\\
\rho([[x, y], \alpha (z)])\beta^{2} & =& \rho(\alpha^{2}(x))\rho(\alpha (y))\rho(z)-  \rho(\alpha^{2}(z))\rho(\alpha (x))\rho(y) \nonumber \\
& + & \rho(\alpha^{2}(y))\rho([z,x]) \beta -\rho(\alpha ([y, z]))\rho(\alpha (x))\beta.
\label{representation H-M}
\end{eqnarray}
We denote a representation by $(V, \rho,\beta)$.
\end{defn}
  \begin{ex}\label{ex:ad}
Let $(A,[-,-],\alpha)$  be a Hom-Malcev algebra. For any integer $s$, the $\alpha^s$-adjoint representation of $(A,[-,-],\alpha)$ on $A$, which we denote by $ad^s$, is defined
for all $x, y\in A$ by
$$ad^s_x (y)=[\alpha^s(x),y].$$
In particular, we write simply $ad$ for $ad^0$.
\end{ex}
\begin{ex}\label{lem:dualrep}
 Let $(V,\rho,\beta)$ be a representation of a regular Hom-Malcev algebra $(A,[-,-],\alpha)$ and let $V^*$ be  the dual vector space of $V$ with $\beta$ is invertible. Let us define $\rho^\star:A\longrightarrow End(V^*)$ by
\begin{align}\label{eq:new1}
 \rho^\star(x)(\xi)&=\rho^*(\alpha(x))\big((\beta^{-2})^*(\xi)\big)=-\langle\xi,\rho(\alpha^{-1}(x))(\beta^{-2}(a))\rangle, \quad\forall x\in A,\xi\in V^*,
\end{align}
where $\langle -,- \rangle$ denote the canonical pairing between the dual  space $V^{*}$ and the vector space $V$. %defined by
%\begin{equation}
%\langle , \rangle:V^* \times V\rightarrow \mathbb{K},\quad a^*(b)= \langle a^*,b\rangle=\langle b,a^*\rangle,\;\;\forall  a^*\in V^*,b\in V.
%\end{equation}
Then $\rho^\star:A\longrightarrow End(V^*)$ is a representation of $(A,[-,-],\alpha)$ on $V^*$ with respect to $(\beta^{-1})^*$, which is called the  dual representation of $(V,\rho,\beta)$.
\end{ex}
\begin{prop} \label{semidirectprduct HomMalcev}
Let $(A,[-, -], \alpha)$ be a Hom-Malcev algebra, $(V, \beta)$ a vector space and $\rho : A \longrightarrow End(V )$ a linear map. Then $(V, \rho, \beta)$ is a representation of $A$ if and only if
$(A \oplus V, [-, -]_{\rho}, \alpha + \beta)$ is a Hom-Malcev algebra, where $[-, -]_{\rho}$ and $\alpha + \beta$ are defined for all  $x, y \in  A,\ a, b \in V$ by
\begin{eqnarray*}
[x + a, y + b]_{\rho} &=& [x, y] + \rho(x)a - \rho(y)b, \\
(\alpha + \beta)(x + a) &=& \alpha(x) + \beta(a).
\end{eqnarray*}
This Hom-Malcev algebra is called the semi-direct product of $(A,[-, -], \alpha)$ and $(V, \beta)$, and denoted by
$A \ltimes_{\rho}^{\alpha, \beta} V$ or simply $A \ltimes V$.
\end{prop}
The notion of $\mathcal{O}$-operator is useful
tool for the construction of the solutions of the classical Yang-Baxter equation. Let $(A, [-, -], \alpha)$
be a Hom-Malcev algebra and $(V,\rho,\beta)$ a representation.
A linear map $T : V \to  A$ is called an $\mathcal{O}$-operator  associated to $(V,\rho,\beta)$   if
for all $a, b \in V,$
\begin{align}\label{ophommalcev}
 \alpha\circ T =  T\circ\beta, \quad \quad [T (a), T (b)] = T \big(\rho(T (a))b - \rho(T (b))a\big).
 \end{align}
 \begin{ex}
 A Rota-Baxter operator of  weight $0$ on a Hom-Malcev algebra $A$ is
just an $\mathcal{O}$-operator associated to the adjoint representation $(A,ad,\alpha)$, that is, $\mathcal{R}$ satisfies for all $x,y\in A$,
$$ \mathcal{R}\circ\alpha = \alpha\circ \mathcal{R}, \quad \quad [\mathcal{R}(x),\mathcal{R}(y)] = \mathcal{R}([\mathcal{R}(x), y] + [x, \mathcal{R}(y)]).$$
\end{ex}
Now we recall the definition of Hom-pre-Malcev algebra and exhibit construction results in terms of
$\mathcal{O}$-operators on Hom-Malcev algebras (for more details see \cite{Fattoum})

\begin{defn}
Let $A$ be a vector space with a bilinear
product $"\cdot"A\times A\to A$ and $\alpha:A\to A$
is a linear map. The triple $(A, \cdot, \alpha)$  is called a $\textbf{Hom-pre-Malcev algebra}$ if for all $x,y,z,t \in A$,
\begin{equation}\label{HPM}
\begin{split}
  &\alpha([y,z]) \cdot \alpha(x \cdot t) + [[x , y] , \alpha(z)] \cdot \alpha^{2}(t) \\
&\qquad + \alpha^{2}(y) \cdot ([x , z] \cdot \alpha (t)) - \alpha^{2}(x) \cdot (\alpha (y)  \cdot (z \cdot t)) + \alpha^{2}(z) \cdot (\alpha (x) \cdot (y \cdot t))=0,\end{split}
\end{equation}
where $[x,y]=x\cdot y-y\cdot x.$
\end{defn}

The identity \eqref{HPM} is equivalent to
$HPM(x, y, z, t) = 0$, where for all $x,y,z,t\in A,$
\begin{equation}\label{HPMexpanded}
\begin{split}
0&=  \alpha(y \cdot z) \cdot \alpha(x \cdot t) - \alpha(z \cdot y) \cdot \alpha(x \cdot t)\\
&\quad+ ((x \cdot y) \cdot \alpha(z)) \cdot \alpha^{2}(t) - ((y \cdot x) \cdot \alpha (z)) \cdot \alpha^{2}(t)\\
 &\quad- (\alpha (z) \cdot (x \cdot y)) \cdot \alpha^{2}(t) + (\alpha (z) \cdot (y \cdot x)) \cdot \alpha^{2}(t)\\
&\quad+ \alpha^{2}(y) \cdot ((x \cdot z) \cdot \alpha (t)) - \alpha^{2}(y) \cdot ((z \cdot x) \cdot \alpha (t))\\
 &\quad- \alpha^{2}(x) \cdot (\alpha (y)  \cdot (z \cdot t)) + \alpha^{2}(z) \cdot (\alpha (x) \cdot (y \cdot t)).
\end{split}
\end{equation}

A Hom-pre-Malcev algebra is said to be a $\textbf{regular Hom-pre-Malcev algebra}$ if $\alpha$ is invertible.
\begin{defn}
 Let  $(A, \cdot, \alpha)$ and  $(A', \cdot', \alpha')$ be two Hom-pre-Malcev algebras. A
linear map $f : A\to A' $ is called a morphism of Hom-pre-Malcev algebras
if, for all $x, y\in A$,
$$f(x)\cdot' f(y) = f(x\cdot y),\qquad f \circ \alpha = \alpha'\circ f. $$
\end{defn}

\begin{prop} \label{prop:HompreMalcevHomMalcevadmis}
 Let $(A, \cdot, \alpha)$ be a Hom-pre-Malcev algebra. Then
 the commutator $[x, y] = x\cdot y - y\cdot x$ defines a Hom-Malcev algebra on $A$.
 It is called the  $\textbf{sub-adjacent Hom-Malcev algebra}$ of
$(A,\cdot,\alpha)$ and $(A, \cdot, \alpha)$ is called the \textbf{compatible~~
Hom-pre-Malcev algebra} of $(A, [-, -]^{C}, \alpha)$.
 The sub-adjacent Hom-Malcev algebra $(A,\cdot,\alpha)$ will be denoted by $A^C$.
 \end{prop}

\begin{prop}\label{hommalcev==>hompremalcev}
 Let $(A,[-,-],\alpha)$ be a Hom-Malcev algebra  and $T : V \to A$ be an $\mathcal{O}$-operator of  $A$
associated to a representation $(V,\rho, \beta)$. Then $(V,\cdot, \beta)$ is a Hom-pre-Malcev algebra with the product $"\cdot"$ defined for all $a, b \in V$ by
\begin{equation}
a\cdot b= \rho(T(a))b.
\end{equation}
\end{prop}

An obvious consequence of Proposition \ref{hommalcev==>hompremalcev} is the construction of a Hom-pre-Malcev algebra in terms of a Rota-Baxter operator of weight zero of a Hom-Malcev
algebra.
\begin{cor}
Let $\mathcal{R}: A \longrightarrow A$ is a Rota-Baxter operator on a Hom-Malcev algebra
$(A,[-,-],\alpha)$.  Then $(A, \cdot, \alpha)$ is a Hom-pre-Malcev algebra, where for all $x, y \in A,$
$$x \cdot y = [\mathcal{R}(x), y].$$
\end{cor}

%\begin{prop}\label{pro:nsc}
%   Let $(A, [-,-],\alpha)$ be a Hom-Malcev algebra and $(V,\rho,\beta)$ be a representation of $A$. Then there exists  a compatible Hom-pre-Malcev algebra structure on $A$ if and only if there is an invertible $\mathcal{O}$-operator on $A$ associated to $(V,\rho,\beta)$ .
%\end{prop}
\section{Representation of Hom-pre-Malcev algebras}

Now, we introduce the definition of representation of Hom-pre-Malcev algebra and give some
related results. we construct the dual representation of a representation of a Hom-pre-Malcev algebra. Also, we give the notion of $\mathcal{O}$-operators  of a Hom-pre-Malcev algebra which is a generalisation of Rota-Baxter operators.
\begin{defn}\label{representation HPM}
 A representation of Hom-pre-Malcev algebra $(A, \cdot, \alpha)$ on a vector space $V$ consists of a triple $(\ell,  r, \beta)$, where
  $\ell: A \to  End(V)$  is a representation of the
sub-adjacent Hom-Malcev algebra $A ^{C}$ on $V$ with respect to $\beta\in End(V )$ and $r: A \to  End(V)$ is a linear map such that for
all  $x,y,z \in A$,
\begin{align}
 &\beta r(x)=r(\alpha(x)) \beta,\label{rep1}\\
\begin{split}
& r(\alpha^{2}(x))\rho(\alpha (y))\rho(z)- r(\alpha (z)\cdot(y\cdot x))\beta^{2} + \ell(\alpha^{2}(y))r(z\cdot x)\beta \\
&\quad \quad+ \ell(\alpha ([y,z]))r(\alpha (x))\beta - \ell(\alpha^{2}(z))r(\alpha (x))\rho(y) = 0,
\end{split}
\label{rep2}\\
\begin{split}
 &\ell(\alpha^{2} (y))\ell(\alpha (z))r(x)-r(\alpha^{2}(x))\rho(\alpha (y))\rho(z) - \ell(\alpha^{2}(z))r(y\cdot x)\beta \\
 &\quad \quad- r(\alpha (z\cdot x))\rho(\alpha (y))\beta+ r([z,y]\cdot\alpha (x))\beta^{2} = 0,
 \end{split}
 \label{rep3}\\
 \begin{split}
   &r(\alpha (y)\cdot(z\cdot x))\beta^{2}+r(\alpha^{2} (x))\rho([y,z])\beta-\ell(\alpha^{2} (y))\ell(\alpha (z)) r(x) \\
    &\quad \quad+r(\alpha (y\cdot x))\rho(\alpha (z))\beta +\ell(\alpha^{2}(z))r(\alpha (x))\rho(y)= 0,
    \end{split}
    \label{rep4}
 \end{align}
 where $\rho(x)=\ell(x)-r(x)$ and $[x,y]= x\cdot y - y\cdot x$.
\end{defn}
We denote a representation of a Hom-pre-Malcev algebra $(A,\cdot,\alpha)$ by $(V,\ell,r, \beta)$. Now, define a linear operation $\cdot_{\ltimes} : \otimes^{2}(A \oplus V ) \longrightarrow (A \oplus V )$ by
$$(x + a) \cdot_{\ltimes} (y + b) = x \cdot y + \ell(x)(b) + r(y)(a), ~~\forall x, y \in A, a, b \in V,$$
and a linear map $\alpha + \beta : A \oplus V \longrightarrow A \oplus V$ by
$$(\alpha + \beta)(x + a) = \alpha(x) + \beta(a), ~~\forall x \in A, a \in V.$$
\begin{prop}[\cite{Fattoum}]\label{semidirectproduct hompreMalcev}
The triple $(A\oplus V, \cdot_{\ltimes}, \alpha+ \beta)$ is a  Hom-pre-Malcev algebra.
\end{prop}

The Hom-pre-Malcev algebra $(A\oplus V, \cdot_{\ltimes}, \alpha+ \beta)$ is
denoted by $A \ltimes_{(\ell,~r)}^{\alpha, \beta}V$, or simply $A \ltimes V$, and called the semi-direct product of the  Hom-pre-Malcev algebra $(A, \cdot, \alpha)$ and
the representation $(V, \ell, r, \beta)$.
\begin{ex}\label{RegRep}
Let $(A,\cdot,\alpha)$ be a  Hom-pre-Malcev algebra.
Define the left multiplication $L :A\to  End(A)$ by  $L_{x}(y)=x\cdot y$ for all $x, y\in A$. Then $(A, L, \alpha)$ is a representation of the Hom-Malcev algebra $A ^{c}$. Moreover, we define the right multiplication
$R : A\to End(A)$ by $R_{x}(y) = y\cdot x$. It is obvious that $(A, L, R,\alpha)$ is a representation of a Hom-pre-Malcev
algebra on itself, which is called $\textbf{the regular representation}$.
\end{ex}
\begin{prop}\label{prop:regular}
Let $(A,\cdot,\alpha)$ be a multiplicative Hom-pre-Malcev algebra. For any integer $s$, define  $L^{s}, R^{s} :A\to  End(A)$ for all $x, y\in A$ by
$$L^{s}_{x} (y)=\alpha^{s}(x)\cdot y,\quad R^{s}_{x}(y)=y\cdot\alpha^{s}(x).$$
Then $(A, L^{s}, R^{s},\alpha)$ is a representation of a Hom-pre-Malcev algebra $(A,\cdot,\alpha)$. Note that $L^0=L$ and $R^0=R$, where $(L,R)$ is the regular representation with  respect $\alpha$ define in Exemple \ref{RegRep}.
\end{prop}
\begin{proof}
For all $x, y, z, t \in A$, we have
$$L ^{s}_{\alpha(x)} \alpha(y) = \alpha^{ s+1}(x)\cdot\alpha(y) = \alpha(\alpha^s (x)\cdot y) = \alpha(L^s
 _{x} y),$$
which implies that $L^s_{\alpha(x)} \circ\alpha = \alpha \circ L^s_{x}$. Similarly, we have $R^s_{\alpha(x)} \circ\alpha = \alpha \circ R^s_{x}$.

By the definition of a Hom-pre-Malcev algebra,
\begin{align*}
    L^{s}_{\alpha[y,z]}L^{s}_{\alpha(x)}\alpha(t) +  L^{s}_{[[x,y],\alpha(z)]}\alpha^{2}(t)&=
    \alpha^{s+1}([y, z]) \cdot (\alpha^{s+1}(x) \cdot \alpha(t)) + \alpha^{s}([[x, y], \alpha(z)])\cdot \alpha^{2}(t)\\
     &= \alpha^{s+2}(y)\cdot (\alpha^{s}([z, x]) \cdot \alpha(t)) +\alpha^{s+2}(x)\cdot (\alpha^{s+1}(y)\cdot (\alpha^{s}(z)\cdot t))\\
      &\quad- \alpha^{s+2}(z)\cdot (\alpha^{s+1}(x)\cdot (\alpha^{s}(
     y)\cdot t))\\
 &=   L^{s}_{\alpha^{2}(y)} L^{s}_{[z,x]}\alpha(t)+ L^{s}_{\alpha^{2}(x)}L^{s}_{\alpha(y)}L^{s}_{z}(t) - L^{s}_{\alpha^{2}(z)}L^{s}_{\alpha(x)}L^{s}_{x}(t) = 0.
\end{align*}
Then $(A, L^{s},\alpha)$  is a representation of the
sub-adjacent Hom-Malcev algebra $A ^{C}$.
Furthermore,
\begin{align*}
     R^{s}_{\alpha^{2}(x)}ad^{s}_{\alpha (y)}ad^{s}_{z}(t)+ L^{s}_{\alpha ([y,z])}R^{s}_{\alpha (x)}\alpha(t)&=[\alpha^{s+1}(y), [\alpha^{s}(z), t]]\cdot \alpha^{s+2}(x)+ \alpha^{s+1}([y,z])\cdot(\alpha(t)\cdot\alpha^{s}(x))\\
     &=\alpha^{s+2}(z)\cdot ([\alpha^{s}(y), t]\cdot\alpha^{s+1}(x))+ \alpha^{2}(t)\cdot (\alpha^{s+1}(z)\cdot\alpha^{s}(y\cdot x))\\
     &\quad-\alpha^{s+2}(y)\cdot (\alpha(t)\cdot\alpha^{s}(z\cdot x))\\
     &= L^{s}_{\alpha^{2}(z)}R^{s}_{\alpha (x)}ad^{s}_{y}(t)+R^{s}_{\alpha (z)\cdot(y\cdot x)}\alpha^{2}(t) + L^{s}_{\alpha^{2}(y)}R^{s}_{z\cdot x}\alpha(t).
\end{align*}
This implies that Eq. \eqref{rep2} holds.
Similarly Eqs. \eqref{rep3} and \eqref{rep4} hold, too.
\end{proof}

\begin{prop}\label{rephompremalcev==rephommalcev}
  Let $(V,\ell,r,\beta)$ be a  representation  of a Hom-pre-Malcev algebra $(A,\cdot,\alpha)$. Then
    $(V,\ell - r,\beta)$ is a representation of $(A, [-, -],\alpha)$.
\end{prop}
Let $(V,\ell,r,\beta)$ be a representation of a Hom-pre-Malcev algebra $(A,\cdot ,\alpha)$. In the sequel, we always assume that $\beta$ is invertible. For all $x\in A,~a\in V,~\xi\in V^*$ define $\ell^*,r^{*}:A\longrightarrow End(V^*)$   as usual by
$$\langle \ell^*(x)(\xi),a\rangle=-\langle\xi,\ell(x)(a)\rangle,\quad \langle r^*(x)(\xi),a\rangle=-\langle\xi,r(x)(a)\rangle .$$
Then define $\ell^\star, r^{\star}:A\longrightarrow End(V^*)$ by
\begin{align}
 \ell^\star(x)(\xi)&=\ell^*(\alpha(x))\big((\beta^{-2})^*(\xi)\big),\label{eq:new2}\\
 r^\star(x)(\xi)&=r^*(\alpha(x))\big((\beta^{-2})^*(\xi)\big)\label{eq:new3}.
\end{align}

\begin{thm}\label{thm:dualrep}
 Let $(V,\ell,r,\beta)$ be a representation of a Hom-pre-Malcev algebra $(A,\cdot,\alpha)$. The quadruple  $(V^*,\ell^\star-r^\star,-r^\star,(\beta^{-1})^*)$ is a representation of $(A,\cdot,\alpha)$. It is called the \textbf{dual representation} of $(V,\ell,r,\beta)$.
\end{thm}
\begin{proof}
Since $(V, \ell, r, \beta)$ is a representation of a Hom-pre-Malcev algebra $(A, \cdot, \alpha)$ and by Proposition \ref{rephompremalcev==rephommalcev},
we deduce that  $(V,\ell - r,\beta)$ is a representation  of the sub-adjacent $A^{C}$.  Moreover,
by Lemma \ref{lem:dualrep},
$ (V,(\ell - r)^{\star}=\ell^{\star}-r^{\star},(\beta^{-1})^{*})$
is a representation of the Hom-Malcev algebra $A^C$.
By  \eqref{rep1} and \eqref{eq:new3},
for all $x\in A,\ \xi\in V^*$,
\begin{eqnarray*}
-r^\star(\alpha(x))((\beta^{-1})^*(\xi))=-r^*(\alpha^{2}(x))(\beta^{-3})^*(\xi)=(\beta^{-1})^*(-r^*(\alpha(x))(\beta^{-2})^*(\xi))=(\beta^{-1})^*(-r^\star(x)(\xi)),
\end{eqnarray*}
which yields $-r^\star\big{(}\alpha(x)\big{)}\circ(\beta^{-1})^*=(\beta^{-1})^*\circ (-r^\star(x)).$
We check only that  $(V^*,\ell^\star-r^\star,-r^\star,(\beta^{-1})^*)$  verifies \eqref{rep2}.
Let  $x,y,z\in A,\xi\in V^*$, $a\in V$ and according to \eqref{rephommalcev1}, \eqref{rep1}, \eqref{rep2} and \eqref{eq:new3},  we have
\begin{align*}
&\Big\langle r^{\star}(\alpha(z)\cdot(y\cdot x))((\beta^{-2})^*(\xi)),a\Big\rangle=\Big\langle r^{\ast}(\alpha^{2}(z)\cdot(\alpha(y)\cdot\alpha( x)))((\beta^{-4})^*(\xi)),a\Big\rangle\\
&=-\Big\langle ((\beta^{-4})^*(\xi)),r(\alpha^{2}(z)\cdot(\alpha(y)\cdot\alpha( x)))(a)\Big\rangle\\
&=-\Big\langle ((\beta^{-4})^*(\xi)),r(\alpha^{3}(x))(\rho(\alpha^{2} (y))(\rho(\alpha(z))(\beta^{-2}(a)))) + \ell(\alpha^{3}(y))(r(\alpha(z)\cdot\alpha( x))(\beta^{-1}(a))) \\
& \quad+ \ell(\alpha ^{2}([y,z]))(r(\alpha^{2} (x))(\beta^{-1}(a))) - \ell(\alpha^{3}(z))(r(\alpha^{2} (x))(\rho(\alpha(y))(\beta^{-2}(a))))\Big\rangle\\
&=-\Big\langle ((\beta^{-6})^*(\xi)),r(\alpha^{5}(x))(\rho(\alpha^{4} (y))(\rho(\alpha^{3}(z)))) + \ell(\alpha^{5}(y))(r(\alpha^{3}(z)\cdot\alpha^{3}( x))(\beta(a))) \\
& \quad+ \ell(\alpha ^{4}([y,z]))(r(\alpha^{4} (x))(\beta(a))) - \ell(\alpha^{5}(z))(r(\alpha^{4} (x))(\rho(\alpha^{3}(y))))\Big\rangle\\
&=-\Big\langle -\rho^{*}(\alpha^{3} (z))(\rho^*(\alpha^{4} (y))( r^* (\alpha^{5}(x))((\beta^{-6})^*(\xi))))+ r^*(\alpha^{2}(z)\cdot\alpha^{2}( x))(\ell^*(\alpha^{4}(y))((\beta^{-5})^*(\xi)))\\
& \quad +r^*(\alpha^{3} (x))(\ell^*(\alpha ^{3}([y,z]))((\beta^{-5})^*(\xi)))+ \rho^*(\alpha^{3}(y))(r^*(\alpha^{4} (x))(\ell^*(\alpha^{5}(z))((\beta^{-6})^*(\xi)))),a\Big\rangle\\
&=-\Big\langle -\rho^{*}(\alpha^{3} (z))(\rho^*(\alpha^{4} (y))( r^{\star} (\alpha^{4}(x))((\beta^{-4})^*(\xi))))+ r^*(\alpha^{2}(z)\cdot\alpha^{2}( x))(\ell^{\star}(\alpha^{3}(y))((\beta^{-3})^*(\xi)))\\
& \quad +r^*(\alpha^{3} (x))(\ell^{\star}(\alpha ^{2}([y,z]))((\beta^{-3})^*(\xi)))+ \rho^*(\alpha^{3}(y))(r^*(\alpha^{4} (x))(\ell^{\star}(\alpha^{4}(z))((\beta^{-4})^*(\xi)))),a\Big\rangle\\
&=-\Big\langle -\rho^{*}(\alpha^{3} (z))(\rho^*(\alpha^{4} (y))( (\beta^{-4})^*(r^{\star} (x)(\xi))))+ r^*(\alpha^{2}(z)\cdot\alpha^{2}( x))((\beta^{-3})^*(\ell^{\star}(y)(\xi)))\\
& \quad +r^*(\alpha^{3} (x))((\beta^{-2})^*(\ell^{\star}([y,z])((\beta^{-1})^*(\xi))))+ \rho^*(\alpha^{3}(y))(r^*(\alpha^{4} (x))((\beta^{-4})^*(\ell^{\star}(z)(\xi)))),a\Big\rangle\\
&=-\langle -\rho^{*}(\alpha^{3} (z))(\rho^{\star}(\alpha^{3} (y))( (\beta^{-2})^*(r^{\star} (x)(\xi))))+ r^{\star}(\alpha(z)\cdot\alpha( x))((\beta^{-1})^*(\ell^{\star}(y)(\xi)))\\
& \quad +r^{\star}(\alpha^{2} (x))(\ell^{\star}([y,z])((\beta^{-1})^*(\xi)))+ \rho^*(\alpha^{3}(y))(r^{\star}(\alpha^{3} (x))((\beta^{-2})^*(\ell^{\star}(z)(\xi)))),a\Big\rangle\\
&=-\langle -\rho^{*}(\alpha^{3} (z))( (\beta^{-2})^*(\rho^{\star}(\alpha (y))(r^{\star} (x)(\xi))))+ r^{\star}(\alpha(z)\cdot\alpha( x))(\ell^{\star}(\alpha(y))((\beta^{-1})^*(\xi)))\\
&\quad +r^{\star}(\alpha^{2} (x))(\ell^{\star}([y,z])((\beta^{-1})^*(\xi)))+ \rho^*(\alpha^{3}(y))((\beta^{-2})^*(r^{\star}(\alpha(x))(\ell^{\star}(z)(\xi)))),a\Big\rangle\\
&=-\langle -\rho^{\star}(\alpha^{2} (z))(\rho^{\star}(\alpha (y))(r^{\star} (x)(\xi)))+ r^{\star}(\alpha(z)\cdot\alpha( x))(\ell^{\star}(\alpha(y))((\beta^{-1})^*(\xi)))\\
&\quad +r^{\star}(\alpha^{2} (x))(\ell^{\star}([y,z])((\beta^{-1})^*(\xi)))+ \rho^{\star}(\alpha^{2}(y))(r^{\star}(\alpha(x))(\ell^{\star}(z)(\xi))),a\Big\rangle,
\end{align*}
which implies that
\begin{align*}
 &r^{\star}(\alpha(z)\cdot(y\cdot x))(\beta^{-2})^* -\rho^{\star}(\alpha^{2} (z))\rho^{\star}(\alpha (y))r^{\star} (x)+ r^{\star}(\alpha(z)\cdot\alpha( x))(\ell^{\star}-r^{\star})(\alpha(y))(\beta^{-1})^*\\
&\quad \quad +r^{\star}(\alpha^{2} (x))(\ell^{\star}-r^{\star})([y,z])(\beta^{-1})^*+ \rho^{\star}(\alpha^{2}(y))r^{\star}(\alpha(x))(\ell^{\star}-r^{\star})(z)=0.
\end{align*}
The other identities can be verified similarly.
Therefore,   $(V^*,\ell^\star-r^\star,-r^\star,(\beta^{-1})^*)$ is a representation of $(A,\cdot,\alpha)$.
\end{proof}
Theorem \ref{thm:dualrep} yields the following result for the dual representation of the regular representation of a regular Hom-pre-Malcev algebra.
\begin{cor}
Let $(A,\cdot,\alpha)$ be a regular Hom-pre-Malcev algebra, and let $A^*$ be the dual of vector space $A$. Then $(A^{*}, ad^{\star} = L^{\star} -R^{\star},-R^{\star},(\alpha^{-1})^{*})$ is
a representation of $(A,\cdot,\alpha)$.
\end{cor}
\begin{prop}
 Let $(V,\ell,r,\beta)$ be a representation of a Hom-pre-Malcev algebra $(A,\cdot,\alpha)$, where $V$ is finite-dimensional and $\beta$ is invertible. Then the dual representation of $(V^*,\ell^\star-r^\star,-r^\star,(\beta^{-1})^*)$ is $(V, \rho, r, \beta)$.
\end{prop}
\begin{proof}
Obviously, $(V^*)^*\simeq V$, $(((\beta^{-1})^*)^{-1})^*=\beta$. By \cite[Lemma 3.7]{CaiSheng}, we obtain
$$(\ell^\star-r^\star)^{\star}=(\rho^\star)^{\star}=\rho.$$
Similarly, we have $(-(-r^\star))^{\star} = r$. Then the dual representation of $(V^*,\ell^\star-r^\star,-r^\star,(\beta^{-1})^*)$ is $(V, \rho, r, \beta)$.
\end{proof}

As in \cite{Bai2,K2}, we rephrase the definition of $\mathcal{O}$-operator in terms of Hom-pre-Malcev
algebras.
   \begin{defn}\label{o-ophpm}
Let $(A, \cdot, \alpha)$ be a Hom-pre-Malcev algebra and $(V,\ell,r,\beta)$ be a representation. A linear map $T : V \to  A $ is called an $\mathcal{O}$-operator associated to $(V,\ell,r,\beta)$
 if $T$ satisfies
 \begin{align}
 T\circ \beta &= \alpha \circ T,\\
 T (a) \cdot T (b)&= T \big(\ell(T (a))b + r(T (b))a\big), \quad\forall a, b \in V.
\end{align}
\end{defn}
\begin{ex}
Let $T:V\to A$ be an $\mathcal{O}$-operator on a pre-Malcev algebra $(A,\cdot)$ with respect to a  representation $(V,\ell,r)$ . A pair $(\phi_{A},\phi_V)$ is an endomorphism of the $\mathcal{O}$-operator $T$ if
\begin{align*}
T\circ \phi_{V}&=\phi_{A}\circ T,\\
\ell(\phi_{A}(x))(\phi_{V}(b))&=\phi_V(\ell(x)(b))\quad \mbox{and}\\
r(\phi_{A}(x))(\phi_{V}(b))&=\phi_V(r(x)(b)) , \quad \mbox{for all } x\in A, ~b\in V.
\end{align*}

 Let us consider the Hom-pre-Malcev algebra $(A,\cdot_{\phi_A},\phi_{A})$ obtained by composition, where the Hom-pre-Malcev product is given by
  $$x\cdot_{\phi_A} y:=\phi_{A}(x\cdot y).$$
If we consider the compositions $\ell_{\phi_V}:=\phi_V\circ \ell$ and $r_{\phi_V}:=\phi_V\circ r$, then the triple $(V,\ell_{\phi_V},r_{\phi_V},\phi_V)$ is a representation of $(A,\cdot_{\phi_A},\phi_{A})$. Moreover,
for all $a,b \in V$,
\begin{align*}
T(a)\cdot_{\phi_A} T(b)&=\phi_{A}(T(a)\cdot T(b))=\phi_{A}\big(T(\ell(T(a))b+r(T(b))a)\big),\\
T\big(\ell_{\phi_V}(T(a))b+r_{\phi_V}(T(b))a\big)&=T\big(\phi_V(\ell(T(a))b+r(T(b))a)\big).
\end{align*}
Clearly, it follows that the map $T:V\rightarrow A$ is an $\mathcal{O}$-operator on Hom-pre-Malcev algebra $(A,\cdot_{\phi_A},\phi_{A})$ with respect to the representation $(V,\ell_{\phi_V},r_{\phi_V},\phi_V)$.
\end{ex}

%%%%%%%%%%%%%%%%%%%%%%%%%%%%%%%%%%%%%%%%%%%%%%%%%%%%%%%%%%%%%%%%%%%%%%%%%%%%%
\section{Hom-M-dendriform algebras} \label{sec:Hom-M-dendriform algebras}
%%%%%%%%%%%%%%%%%%%%%%%%%%%%%%%%%%%%%%%%%%%%%%%%%%%%%%%%%%%%%%%%%%%%%%%%%%%%%%
 The goal of this section is to introduce the notion of Hom-M-dendriform algebras  which
are the dendriform version of Hom-Malcev algebras, and study the relationships with  Hom-pre-Malcev algebras and Hom-Malcev algebras  in terms of the $\mathcal{O}$-operators of Hom-pre-Malcev algebras.
 \subsection{Definitions and some basic properties}
\begin{defn}
Hom-M-dendriform algebra is a vector space $A$ endowed with two
bilinear products $\blacktriangleright, \blacktriangleleft: A\times A\to A$ and a linear map
$\alpha: A \to A$
such that for all $x, y, z, t \in A$ and
\begin{align}
&x\cdot y=x \blacktriangleleft y+ x\blacktriangleright y,   \label{horizontal} \\
&x \diamond y= x \blacktriangleleft y -y \blacktriangleright x, \label{vertical} \\
&[x, y]=x\cdot y-y \cdot x=x\diamond y-y\diamond x, \label{horizontal-vertical}
\end{align}
the following identities are satisfied:
\begin{align}\label{dend1}
&
\begin{array}{l}
(\alpha(z)\diamond(y\diamond x))\blacktriangleright \alpha^{2}(t)- \alpha^{2}(x)\blacktriangleright (\alpha (y)\cdot(z\cdot t))+\alpha^{2}(z)\blacktriangleleft(\alpha(x)\blacktriangleright(y\cdot t)) \\
\quad + \alpha([y,z])\blacktriangleleft \alpha(x\blacktriangleright t)
-\alpha^{2}(y)\blacktriangleleft((z\diamond x)\blacktriangleright \alpha(t)) = 0,
\end{array}
\\
\label{dend2}
&
\begin{array}{l}
\alpha^{2}(z)\blacktriangleleft(\alpha(x)\blacktriangleleft(y\blacktriangleright t))
-(\alpha (z)\diamond(x\diamond y))\blacktriangleright \alpha^{2}(t)- \alpha^{2}(x)\blacktriangleleft (\alpha (y)\blacktriangleright(z\cdot t)) \\
\quad
 - \alpha(z\diamond y)\blacktriangleright \alpha(x\cdot t)+ \alpha^{2}(y)\blacktriangleright([x,z]\cdot \alpha(t))= 0,
\end{array} \\
\label{dend3}
&
\begin{array}{l}
\alpha^{2}(z)\blacktriangleleft(\alpha(x)\cdot(y\cdot t))+([x,y]\diamond\alpha(z))\blacktriangleright \alpha^{2}(t)- \alpha^{2}(x)\blacktriangleleft(\alpha(y)\blacktriangleleft(z\blacktriangleright t)) \\
\quad
+ \alpha(y\diamond z)\blacktriangleright \alpha(x\cdot t) + \alpha^{2}(y)\blacktriangleleft((x\diamond z)\blacktriangleright \alpha(t))= 0,
\end{array} \\
\label{dend4}
&
\begin{array}{l}
[[x,y],\alpha (z)]\blacktriangleleft \alpha^{2}(t)-\alpha^{2} (x)\blacktriangleleft(\alpha (y)\blacktriangleleft (z \blacktriangleleft t)) + \alpha^{2}(z)\blacktriangleleft(\alpha (x)\blacktriangleleft (y\blacktriangleleft t)) \\
\quad
+ \alpha ([y,z])\blacktriangleleft\alpha (x \blacktriangleleft t)+ \alpha ^{2}(y)\blacktriangleleft([x,z]\blacktriangleleft\alpha(t))= 0.
\end{array}
\end{align}

A Hom-M-dendriform algebra $(A, \blacktriangleright,\blacktriangleleft, \alpha)$ is said to be \textbf{regular}  if $\alpha$ is invertible, and \textbf{multiplicative} if $\alpha(x \blacktriangleright y) = \alpha(x)
\blacktriangleright\alpha(y)$ and $\alpha(x \blacktriangleleft y) = \alpha(x)
\blacktriangleleft\alpha(y)$.
\end{defn}
Let  $(A_{1}, \blacktriangleright_{1},\blacktriangleleft_{1}, \alpha_{1}))$ and  $(A_{2}, \blacktriangleright_{2},\blacktriangleleft_{2}, \alpha_{2})$ be two Hom-M-dendriform algebras. A weak morphism $f : A_{1}\longrightarrow A_{2} $ is a linear map such that
$$
  f \blacktriangleright_1  = \blacktriangleright_2  f ^{\otimes 2}, \qquad
  f \blacktriangleleft_1   = \blacktriangleleft_2 f ^{\otimes 2}.
$$
A morphism  $f : A_1\longrightarrow A_2$  is a weak morphism such that $f\alpha_1=  \alpha_2 f.$

Now, we provide a way to construct Hom-M-dendriform algebras starting
from an M-dendriform algebra and a weak self-morphism.
\begin{prop}
Let $(A, \blacktriangleright,\blacktriangleleft )$ be an M-dendriform algebra and $\alpha : A \to A$ be a weak morphism. Then
$A_{\alpha}=(A, \blacktriangleright_{\alpha}=\alpha\blacktriangleright, ~~\blacktriangleleft_{\alpha}=\alpha\blacktriangleleft,~~ \alpha)$ is a Hom-M-dendriform algebra.

Moreover, if $\alpha$ is a morphism, then $A_{\alpha}$ is a multiplicative Hom-M-dendriform algebra.

\end{prop}
\begin{proof}
We only prove that $(A, \blacktriangleright_{\alpha}, \blacktriangleleft_{\alpha}, \alpha)$ satisfies the first Hom-M-dendriform identity. The other identities for $A_{\alpha}$ being a Hom-M-dendriform algebra can be verified
similarly. In fact, for any $x, y,z,t \in A$, we know that
\begin{multline*}
(z\diamond(y\diamond x))\blacktriangleright t-x\blacktriangleright (y\cdot(z\cdot t))+z\blacktriangleleft(x\blacktriangleright(y\cdot t))+[y,z]\blacktriangleleft(x\blacktriangleright t) -y\blacktriangleleft((z\diamond x)\blacktriangleright t)=  0.
\end{multline*}
Now applying $\alpha^{3}$ to the previous identity, we obtain
\begin{multline*}
(\alpha^{3}(z)\diamond(\alpha^{3}(y)\diamond \alpha^{3}(x)))\blacktriangleright \alpha^{3}(t)-\alpha^{3}(x)\blacktriangleright (\alpha^{3}(y)\cdot(\alpha^{3}(z)\cdot \alpha^{3}(t)))+ \alpha^{3}(z)\blacktriangleleft(\alpha^{3}(x)\blacktriangleright(\alpha^{3}(y)\cdot \alpha^{3}(t)))\\
\quad\quad\quad+\alpha^{3}([y,z])\blacktriangleleft(\alpha^{3}(x)\blacktriangleright \alpha^{3}(t))  -\alpha^{3}(y)\blacktriangleleft((\alpha^{3}(z)\diamond \alpha^{3}(x))\blacktriangleright \alpha^{3}(t))=  0.
\end{multline*}
That is
\begin{multline*}
(\alpha(z)\diamond_{\alpha} (y\diamond_{\alpha}  x))\blacktriangleright_{\alpha}  \alpha^{2}(t)- \alpha^{2}(x)\blacktriangleright_{\alpha}  (\alpha (y)\cdot_{\alpha} (z\cdot_{\alpha}  t))+\alpha^{2}(z)\blacktriangleleft_{\alpha} (\alpha(x)\blacktriangleright_{\alpha} (y\cdot_{\alpha}  t)) \\
 + \alpha([y,z]_{\alpha} )\blacktriangleleft_{\alpha}  \alpha(x\blacktriangleright_{\alpha}  t)-\alpha^{2}(y)\blacktriangleleft_{\alpha} ((z\diamond_{\alpha}  x)\blacktriangleright_{\alpha}  \alpha(t))=0.
\end{multline*}
Hence  $A_{\alpha}$ is a Hom-M-dendriform algebra.

Moreover, using the multiplicativity of $\alpha$, we obtain
\begin{align*}
&(\alpha(z)\diamond_{\alpha} (y\diamond_{\alpha}  x))\blacktriangleright_{\alpha}  \alpha^{2}(t)- \alpha^{2}(x)\blacktriangleright_{\alpha}  (\alpha (y)\cdot_{\alpha} (z\cdot_{\alpha}  t))+\alpha^{2}(z)\blacktriangleleft_{\alpha} (\alpha(x)\blacktriangleright_{\alpha} (y\cdot_{\alpha}  t)) \\
&+ \alpha([y,z]_{\alpha} )\blacktriangleleft_{\alpha}  \alpha(x\blacktriangleright_{\alpha}  t)-\alpha^{2}(y)\blacktriangleleft_{\alpha} ((z\diamond_{\alpha}  x)\blacktriangleright_{\alpha}  \alpha(t))\\
=&\alpha^{3}\Big((z\diamond(y\diamond x))\blacktriangleright t-x\blacktriangleright (y\cdot(z\cdot t))+z\blacktriangleleft(x\blacktriangleright(y\cdot t))+[y,z]\blacktriangleleft(x\blacktriangleright t)\Big.
\ -y\blacktriangleleft((z\diamond x)\blacktriangleright t)\Big) =  0,
\end{align*}
Which implies \eqref{dend1}. Similarly, we obtain the identities \eqref{dend2}-\eqref{dend4}.
\end{proof}
\begin{rem}
More generally,
let $(A, \blacktriangleright , \blacktriangleleft, \alpha)$ be a Hom-M-dendriform algebra and $\gamma:
A\rightarrow A$ a  Hom-M-dendriform algebra morphism. Define new multiplications on $A$, for all $x, y\in A$, by
\begin{eqnarray*}
&&x\blacktriangleright'y=\gamma(x)\blacktriangleright \gamma(y) \;\;\;\;\;and\;\;\;\;\;
x\blacktriangleleft 'y=\gamma(x)\blacktriangleleft\gamma(y).
\end{eqnarray*}
Then  $(A, \blacktriangleright ', \blacktriangleleft ', \alpha  \gamma)$
is a multiplicative  Hom-M-dendriform algebra.
\end{rem}

\begin{thm}\label{product}
Let $(A, \blacktriangleright,\blacktriangleleft, \alpha)$ be a Hom-M-dendriform algebra. Then,
the product given by   \eqref{horizontal} $($resp. \eqref{vertical}$)$ defines a Hom-pre-Malcev algebra $(A,\cdot, \alpha)$ $($resp. $(A,\diamond, \alpha)$ $)$, called the associated horizontal $($resp. vertical$)$ Hom-pre-Malcev algebras.
\end{thm}
\begin{proof} The identity \eqref{HPM} holds since, for any $x,y,z,t\in A,$
\begin{align*}
&\alpha([y, z])\cdot\alpha(x\cdot t)+[[x,y],\alpha(z)] \cdot \alpha^{2}(t)
+ \alpha^{2}(y) \cdot ([x,z] \cdot \alpha (t))
-\alpha^{2}(x) \cdot (\alpha (y) \cdot (z \cdot t))  \\
&\hspace{9cm} + \alpha^{2}(z) \cdot (\alpha (x) \cdot (y \cdot t))\\
&=\alpha([y, z])\blacktriangleleft\alpha(x\blacktriangleleft t)+\alpha([y,z])\blacktriangleleft\alpha( x\blacktriangleright t) + \alpha([y, z])\blacktriangleright \alpha(x\blacktriangleleft t)+\alpha([y,z])\blacktriangleright \alpha(x\blacktriangleright t)\\
& \quad+[[x,y],\alpha(z)]\blacktriangleleft\alpha^{2}(t)  +[[x,y],\alpha(z)]\blacktriangleright\alpha^{2}(t)+  \alpha^{2}(y)\blacktriangleleft([x,z]\blacktriangleleft\alpha(t))+ \alpha^{2}(y)\blacktriangleleft([x,z]\blacktriangleright\alpha(t))\\
& \quad+\alpha^{2}(y)\blacktriangleright([x,z]\blacktriangleleft\alpha(t))+\alpha^{2}(y)\blacktriangleright([x,z]\blacktriangleright\alpha(t))+\alpha^{2}(z)\blacktriangleleft(\alpha(x)\blacktriangleleft(y\blacktriangleleft t))
\\
&\hspace{9cm}
+\alpha^{2}(z)\blacktriangleleft(\alpha(x)\blacktriangleleft(y\blacktriangleright t))\\
&\quad+ \alpha^{2}(z)\blacktriangleleft(\alpha(x)\blacktriangleright(y\blacktriangleleft t))+\alpha^{2}(z)\blacktriangleleft(\alpha(x)\blacktriangleright(y\blacktriangleright t))
+\alpha^{2}(z)\blacktriangleright(\alpha(x)\blacktriangleleft(y\blacktriangleleft t))
\\
&\hspace{9cm}+ \alpha^{2}(z)\blacktriangleright(\alpha(x)\blacktriangleleft(y\blacktriangleright t))\\
& \quad
 + \alpha^{2}(z)\blacktriangleright(\alpha(x)\blacktriangleright(y\blacktriangleleft t))
 +\alpha^{2}(z)\blacktriangleright(\alpha(x)\blacktriangleright(y\blacktriangleright t))-\alpha^{2}(x)\blacktriangleleft(\alpha(y)\blacktriangleleft( z\blacktriangleleft t))
\\
&\hspace{9cm}
-\alpha^{2}(x)\blacktriangleleft(\alpha(y)\blacktriangleleft(z\blacktriangleright t))
\\
& \quad-\alpha^{2}(x)\blacktriangleleft(\alpha(y)\blacktriangleright(z\blacktriangleleft t))-\alpha^{2}(x)\blacktriangleleft(\alpha(y)\blacktriangleright(z\blacktriangleright t))
-\alpha^{2}(x)\blacktriangleright(\alpha(y)\blacktriangleleft(z\blacktriangleleft t))
\\
&\hspace{9cm}
-\alpha^{2}(x)\blacktriangleright(\alpha(y)\blacktriangleleft(z\blacktriangleright t))\\
&\quad-\alpha^{2}(x)\blacktriangleright(\alpha(y)\blacktriangleright(z\blacktriangleleft t))
-\alpha^{2}(x)\blacktriangleright(\alpha(y)\blacktriangleright(z\blacktriangleright t))\\
&=\alpha([y, z])\blacktriangleleft\alpha(x\blacktriangleright t)+(\alpha(z)\diamond(y\diamond x))\blacktriangleright\alpha^{2}(t)-\alpha^{2}(y)\blacktriangleleft((z\diamond x)\blacktriangleright\alpha(t))+\alpha^{2}(z)\blacktriangleleft(\alpha(x)\blacktriangleright(y\cdot t))\\
& \quad-\alpha^{2}(x)\blacktriangleright(\alpha(y)\cdot(z\cdot t))
-\alpha(z\diamond y)\blacktriangleright \alpha(x\cdot t)-(\alpha(z)\diamond(x\diamond y))\blacktriangleright\alpha^{2}(t)+\alpha^{2}(y)\blacktriangleright([x,z]\cdot\alpha(t))
\\
& \quad+\alpha^{2}(z)\blacktriangleleft(\alpha(x)\blacktriangleleft(y\blacktriangleright t))-\alpha^{2}(x)\blacktriangleleft(\alpha(y)\blacktriangleright(z\cdot t))+\alpha(y\diamond z)\blacktriangleright \alpha(x\cdot t)
+([x,y]\diamond\alpha(z))\blacktriangleright\alpha^{2}(t)\\
&\quad+ \alpha^{2}(y)\blacktriangleleft((x\diamond z)\blacktriangleright\alpha(t))+\alpha^{2}(z)\blacktriangleright(\alpha(x)\cdot(y\cdot t))
-\alpha^{2}(x)\blacktriangleleft(\alpha(y)\blacktriangleleft(z\blacktriangleright t))+[[x,y],\alpha(z)]\blacktriangleleft\alpha^{2}(t)\\
& \quad + \alpha([y, z])\blacktriangleleft \alpha(x\blacktriangleleft t)
+\alpha^{2}(y)\blacktriangleleft([x,z]\blacktriangleleft\alpha(t))+ \alpha^{2}(z)\blacktriangleleft(\alpha(x)\blacktriangleleft(y\blacktriangleleft t)) \\
&\quad -\alpha^{2}(x)\blacktriangleleft(\alpha(y)\blacktriangleleft( z\blacktriangleleft t))=0.
\qedhere
\end{align*}
\end{proof}
\begin{rem}
 The associated horizontal Hom-pre-Malcev algebra and vertical Hom-pre-Malcev algebra, $(A,\cdot, \alpha)$ and $(A,\diamond, \alpha)$, of a Hom-M-dendriform algebra $(A, \blacktriangleright,\blacktriangleleft, \alpha)$ have the same associated Hom-Malcev algebras $(A,[-,-],\alpha)$ defined using \eqref{horizontal-vertical}, called the associated Hom-Malcev algebra of the  Hom-M-dendriform algebra $(A, \blacktriangleright,\blacktriangleleft, \alpha)$.
\end{rem}
Hom-M-dendriform algebras are closely related to representation for Hom-pre-Malcev algebras.
\begin{prop}\label{prop:bimodule}
Let $(A,\blacktriangleright,\blacktriangleleft,\alpha)$ be a Hom-M-dendriform algebra.
Let $L_{\blacktriangleleft}$ and $R_{\blacktriangleright}$ be the left and right multiplication operators corresponding respectively to the two operations $\blacktriangleright$ and $\blacktriangleleft$.
Then $(A,L_{\blacktriangleleft},R_{\blacktriangleright},\alpha)$ is a representation of its associated horizontal Hom-pre-Malcev algebra $(A,\cdot,\alpha)$.
\end{prop}
\begin{proof}If $(A,\blacktriangleright,\blacktriangleleft,\alpha)$ is a Hom-M-dendriform algebra, then for all
$x, y \in A$, we have
\begin{equation*}
     L_{\blacktriangleleft}(\alpha(x))\alpha(y) = \alpha(x) \blacktriangleleft \alpha(y) = \alpha(x \blacktriangleleft y) = \alpha(L_{\blacktriangleleft}(x)y),
\end{equation*}
which implies that $L_{\blacktriangleleft}(\alpha(x))\alpha= \alpha L_{\blacktriangleleft}(x)$. Similarly, we have $R_{\blacktriangleright}(\alpha(x)) \alpha = \alpha  R_{\blacktriangleright}(x).$

For any $x,y,z,t \in A$, by \eqref{dend4}, we have
\begin{align*}
& L_\blacktriangleleft[[x,y],\alpha (z)]\alpha^{2}(t) - L_\blacktriangleleft(\alpha^{2} (x))L_\blacktriangleleft(\alpha (y))L_\blacktriangleleft(z)(t) + L_\blacktriangleleft(\alpha^{2}(z))L_\blacktriangleleft(\alpha (x))L_\blacktriangleleft(y)(t)\\
& \quad + L_\blacktriangleleft(\alpha ([y,z]))L_\blacktriangleleft(\alpha (x))\alpha(t) + L_\blacktriangleleft(\alpha ^{2}(y))L_\blacktriangleleft([x,z])\alpha(t)\\
& =  [[x,y],\alpha (z)]\blacktriangleleft \alpha^{2}(t)-\alpha^{2} (x)\blacktriangleleft(\alpha (y)\blacktriangleleft (z \blacktriangleleft t)) + \alpha^{2}(z)\blacktriangleleft(\alpha (x)\blacktriangleleft (y\blacktriangleleft t))\\
& \quad + \alpha ([y,z])\blacktriangleleft\alpha (x \blacktriangleleft t)+ \alpha ^{2}(y)\blacktriangleleft([x,z]\blacktriangleleft\alpha(t))= 0.
\end{align*}
Moreover, by \eqref{dend1},
\begin{align*}
& R_\blacktriangleright(\alpha^{2}(x))L_\diamond(\alpha (y))L_\diamond(z)(t)- R_\blacktriangleright(\alpha (z)\cdot(y\cdot x))\alpha^{2}(t) + L_\blacktriangleleft(\alpha^{2}(y))R_\blacktriangleright(z\cdot x)\alpha(t)\\
 & \quad  + L_\blacktriangleleft(\alpha ([y,z]))R_\blacktriangleright(\alpha (x))\alpha(t)- L_\blacktriangleleft(\alpha^{2}(z))R_\blacktriangleright(\alpha (x))L_\diamond(y)(t)\\
& =  (\alpha (y)\diamond (z\diamond t))\blacktriangleright\alpha^{2}(x)-\alpha^{2}(t)\blacktriangleright(\alpha (z)\cdot(y\cdot x))  + \alpha^{2}(y)\blacktriangleleft(\alpha(t)\blacktriangleright(z\cdot x))\\
 & \quad + \alpha ([y,z])\blacktriangleleft \alpha (x\blacktriangleright t) - \alpha^{2}(z)\blacktriangleleft((y\diamond t)\blacktriangleright\alpha (x))= 0.
 \end{align*}
 Similarly, we have
\begin{align*}
& R_\blacktriangleright(\alpha^{2}(x))L_\diamond(\alpha (y))L_\diamond(z)(t)- R_\blacktriangleright(\alpha (z)\cdot(y\cdot x))\alpha^{2}(t)  \\
& \quad + L_\blacktriangleleft(\alpha^{2}(y))R_\blacktriangleright(z\cdot x)\alpha(t)+ L_\blacktriangleleft(\alpha ([y,z]))R_\blacktriangleright(\alpha (x))\alpha(t) - L_\blacktriangleleft(\alpha^{2}(z))R_\blacktriangleright(\alpha (x))L_\diamond(y)(t) = 0, \\  & L_\blacktriangleleft(\alpha^{2} (y))L_\blacktriangleleft(\alpha (z))R_\blacktriangleright(x)(t)-R_\blacktriangleright(\alpha^{2}(x))L_\diamond(\alpha (y))L_\diamond(z)(t) \\
& \quad - L_\blacktriangleleft(\alpha^{2}(z))R_\blacktriangleright(y\cdot x)\alpha(t) - R_\blacktriangleright(\alpha (z\cdot x))L_\diamond(\alpha (y))\alpha(t)+ R_\blacktriangleright([z,y]\cdot\alpha (x))\alpha^{2}(t) = 0.
\end{align*}
 Thus, $(A,L_{\blacktriangleleft},R_{\blacktriangleright},\alpha)$ is a representation of Hom-pre-Malcev algebra $(A,\cdot,\alpha)$.
\end{proof}

\begin{cor}
Let $(A,\blacktriangleright,\blacktriangleleft, \alpha)$ be a regular Hom-M-dendriform algebra and let $A^*$ be the dual of $A$. Then $(A^{*}, L^{\star}_{\blacktriangleleft}-R^{\star}_{\blacktriangleright},-R^{\star}_{\blacktriangleright},(\alpha^{-1})^{*})$ is a representation of the associated horizontal Hom-pre-Malcev
algebra $(A,\cdot, \alpha)$ and
$(A^{*}, L^{\star}_{\blacktriangleright}+L^{\star}_{\blacktriangleleft},L^{\star}_{\blacktriangleleft},(\alpha^{-1})^{*})$ is a representation of the associated vertical Hom-pre-Malcev algebra $(A,\diamond, \alpha).$
\end{cor}
\begin{proof}
   It follows from Proposition \ref{prop:bimodule}  and Theorem \ref{thm:dualrep}.
\end{proof}
\begin{prop}
 Let $(A,\blacktriangleright,\blacktriangleleft,\alpha)$ be a Hom-M-dendriform algebra. With two binary operations $\blacktriangleright ^{t}, \blacktriangleleft ^{t}:
A\otimes A \to A$ defined for all $x, y\in A$ by
\begin{equation}\label{transpose}
 x \blacktriangleright ^{t} y = -y \blacktriangleright x, \quad
 x \blacktriangleleft ^{t} y = x \blacktriangleleft y,
\end{equation}
$(A,\blacktriangleright^{t},\blacktriangleleft^{t},\alpha)$ is a Hom-M-dendriform algebra.

Its  associated horizontal Hom-pre-Malcev algebra is the associated vertical Hom-pre-Malcev algebra $(A, \diamond,\alpha)$ of $(A, \blacktriangleright,\blacktriangleleft,\alpha)$, and its associated vertical Hom-pre-Malcev
algebra is the associated horizontal Hom-pre-Malcev algebra $(A,\cdot,\alpha)$ of $(A,\blacktriangleright,\blacktriangleleft,\alpha)$, that is,
\begin{align}
\label{transpose point} x\cdot^t y &=x\blacktriangleleft^t y+x\blacktriangleright^t y=x\blacktriangleleft
  y-y\blacktriangleright x=x\diamond y,\\
\label{tanspose diamond} x\diamond^t y &=x\blacktriangleleft^t y-y\blacktriangleright^t x=x\blacktriangleleft y+x\blacktriangleright y=x\cdot y,\\
\begin{split} [x,y]^{t} &=x\blacktriangleleft^t y+x\blacktriangleright^t y-y\blacktriangleleft^t x-y\blacktriangleright^t x \\
 &=x\blacktriangleleft y-y\blacktriangleright x-y\blacktriangleleft x+x\blacktriangleright y=[x,y].
\label{transpose bracket}
\end{split}
\end{align}
\end{prop}
\begin{proof}
By \eqref{transpose point}-\eqref{transpose bracket}, for all $x,y,z,t\in A$,
\begin{align*}
&[[x,y]^{t},\alpha (z)]^{t}\blacktriangleleft^{t} \alpha^{2}(t)-\alpha^{2} (x)\blacktriangleleft^{t}(\alpha (y)\blacktriangleleft^{t} (z \blacktriangleleft^{t} t)) + \alpha^{2}(z)\blacktriangleleft^{t}(\alpha (x)\blacktriangleleft^{t} (y\blacktriangleleft^{t} t))\\
&+\alpha ([y,z]^{t})\blacktriangleleft^{t}\alpha (x \blacktriangleleft^{t} t)+ \alpha ^{2}(y)\blacktriangleleft^{t}([x,z]^{t}\blacktriangleleft^{t}\alpha(t))\\
=&[[x,y],\alpha (z)]\blacktriangleleft \alpha^{2}(t)-\alpha^{2} (x)\blacktriangleleft(\alpha (y)\blacktriangleleft (z \blacktriangleleft t)) + \alpha^{2}(z)\blacktriangleleft(\alpha (x)\blacktriangleleft (y\blacktriangleleft t))\\
&+ \alpha ([y,z])\blacktriangleleft\alpha (x \blacktriangleleft t)+ \alpha ^{2}(y)\blacktriangleleft([x,z]\blacktriangleleft\alpha(t)).
\end{align*}
Similarly, \eqref{dend1}$^{t}$=\eqref{dend1}, \eqref{dend2}$^{t}$=\eqref{dend2}  and \eqref{dend3}$^{t}$=\eqref{dend3}. Thus, $(A,\blacktriangleright^t,\blacktriangleleft^t,\alpha)$ is a Hom-M-dendriform algebra.
\end{proof}
\begin{defn} Let $(A,\blacktriangleright,\blacktriangleleft,\alpha)$ be a Hom-M-dendriform algebra. The Hom-M-dendri\-form algebra $(A,\blacktriangleright^{t},\blacktriangleleft^{t},\alpha)$
given by \eqref{transpose} is called the transpose of $(A,\blacktriangleright,\blacktriangleleft,\alpha)$.
\end{defn}
%%%%%%%%%%%%%%%%%%%%%%%%%%%%%%
\subsection{Hom-M-dendriform algebras and \texorpdfstring{$\mathcal{O}$}{}-operators of Hom-pre-Malcev algebras}
%%%%%%%%%%%%%%%%%%%%%%%%%%%%%%%
Examples of Hom-M-dendriform algebras can be constructed from Hom-pre-Malcev algebras
with $\mathcal{O}$-operators.
For brevity, we only give the study involving the associated
horizontal Hom-pre-Malcev algebras.
\begin{prop}
 Let $(A, \cdot, \alpha)$ be a Hom-pre-Malcev algebra and $(V, \ell, r, \beta)$ be a representation of $A$.  Let $T$ be an $\mathcal{O}$-operator of $(A, \cdot, \alpha)$  associated to $(V, \ell, r, \beta)$.
Then $(V, \blacktriangleright, \blacktriangleleft, \beta)$ is a Hom-M-dendriform algebra, where
for all $a, b \in V$,
\begin{equation}\label{hpm==>dend}
~a \blacktriangleright b = r(T(b))a, \quad a\blacktriangleleft b = \ell(T(a))b.
\end{equation}
Therefore, there is a Hom-pre-Malcev algebra on $V$ given in Theorem \ref{product} as the associated horizontal Hom-pre-Malcev algebra of $(V,\blacktriangleright,\blacktriangleleft, \beta)$, and $T$ is a morphism of Hom-pre-Malcev algebras. Moreover, $T(V) = \{T(v) \mid \; v \in V\} \subseteq A $ is a Hom-pre-Malcev subalgebra of $(A, \cdot,\alpha)$, and there is an induced Hom-M-dendriform algebraic structure on $(T(V ),\triangleright,\triangleleft,\alpha)$ given for all $a, b \in V$ by
\begin{align*}
    T(a) \triangleright T(b)=T(a\blacktriangleright b),\quad  T(a) \triangleleft T(b)=T(a \blacktriangleleft b).
\end{align*}
Its corresponding associated horizontal Hom-pre-Malcev algebraic
structure on $T(V )$ is just the subalgebra of the Hom-pre-Malcev $(A, \cdot,\alpha)$, and $T$ is a
homomorphism of Hom-M-dendriform algebras.
\end{prop}
\begin{proof}
For any $a,b \in V$, define $\c_V , \diamond_V ,[-,-]_V:\otimes^3V \to V$ by
\begin{align*}
& a\c_V b=a\blacktriangleleft b+a\blacktriangleright b,\\
& a\diamond_V b,= a\blacktriangleleft b-b \blacktriangleright a,\\
&[a,b]_V= a\c_V b- b\c_V a =a\diamond_V b-b\diamond_V a.
\end{align*}
Using identity \eqref{o-ophpm}, we have
\begin{align*}
T(a\c_V b) &=T(a\blacktriangleleft b+a\blacktriangleright b)\\
&=T(\ell(T(a))b+ r(T(b))a)=T(a)\cdot T(b),
\end{align*}
and
$$T([a,b]_V)=T(a\c_V b- b\c_V a)
=T(a)\cdot T(b)-T(b)\cdot T(a)=[T(a),T(b)].$$
Hence, for any $a, b, c, d \in V$, we have
\begin{align*}
   & (\beta(c)\diamond_{V}(b\diamond_{V} a))\blacktriangleright \beta^{2}(d)- \beta^{2}(a)\blacktriangleright (\beta (b)\cdot_{V}(c\cdot_{V} d))+\beta^{2}(c)\blacktriangleleft(\beta(a)\blacktriangleright(b\cdot_{V} d))\\
  &\quad + \beta([b,c]_{V})\blacktriangleleft \beta(a\blacktriangleright d)-\beta^{2}(b)\blacktriangleleft((c\diamond_{V} a)\blacktriangleright \beta(d))\\
 & = r(T(\beta^{2}(d)))(\beta(c)\diamond_{V}(b\diamond_{V} a)) -r(T(\beta (b)\cdot_{V}(c\cdot_{V} d)))\beta^{2}(a)+ \ell(T(\beta^{2}(c)))r(T(b\cdot_{V} d))\beta(a)\\
  &\quad + \ell(T(\beta([b,c]_{V})))\beta(r(T(d))a)- \ell(T(\beta^{2}(b)))r(T(\beta(d)))(c\diamond_{V} a)\\
 & = r(\alpha^{2}(T(d)))\rho(\alpha(T(c)))\rho(T(b))a -r(\alpha(T(b))\cdot(T(c)\cdot T(d)))\beta^{2}(a)+\ell(\alpha^{2}(T(c)))r(T(b)\cdot T(d))\beta(a)\\
 &\quad +\ell(\alpha([T(b),T(c)]))r(\alpha(T(d)))\beta(a)- \ell(\alpha^{2}(T(b)))r(\alpha(T(d)))\rho(T(c))a =0.
\end{align*}
This implies that \eqref{dend1} holds.
Moreover, \eqref{dend2} holds. Indeed,
\begin{align*}
&\beta^{2}(c)\blacktriangleleft(\beta(a)\blacktriangleleft(b\blacktriangleright d))-(\beta (c)\diamond_{V}(a\diamond_{V} b))\blacktriangleright \beta^{2}(d)- \beta^{2}(a)\blacktriangleleft (\beta (b)\blacktriangleright(c\cdot_{V} d)) \\
& \quad - \beta(c\diamond_{V} b)\blacktriangleright \beta(a\cdot_{V} d)+ \beta^{2}(b)\blacktriangleright([a,c]_{V}\cdot_{V} \beta(d))\\
& = \ell(T(\beta^{2} (c)))\ell(T(\beta (a))r(T(d))b-r(T(\beta^{2}(d))(\beta (c)\diamond_{V}(a\diamond_{V} b)) - \ell(T(\beta^{2}(a)))r(T(c\cdot_{V} d))\beta(b) \\
& \quad - r(T(\beta (a\cdot_{V} b)))\beta(c\diamond_{V} b)+ r(T([a,c]_{V}\cdot_{V}\beta (x)))\beta^{2}(b)\\
& = \ell(\alpha^{2}(T(c)))\ell(\alpha (T(a)))r(T(d))b-r(\alpha^{2}(T(d)))\rho(\alpha(T (c)))\rho(T(a))b - \ell(\alpha^{2}(T(a)))r(T(c)\cdot T(d))\beta(b) \\
& \quad - r(\alpha (T(a)\cdot T(b)))\rho(\alpha (T(c))\beta(b)+ r([T(a),T(c)]\cdot\alpha (T(d)))\beta^{2}(b)
= 0.
\end{align*}
To prove identity \eqref{dend3}, we compute as follows
\begin{align*}
& \beta^{2}(c)\blacktriangleright(\beta(a)\cdot_{V}(b\cdot_{V}d))
+([a,b]_{V}\diamond_{V}\beta(c))\blacktriangleright \beta^{2}(d)- \beta^{2}(a)\blacktriangleleft(\beta(b)\blacktriangleleft(c\blacktriangleright d)) \\
&\quad + \beta(b\diamond_{V} c)\blacktriangleright \beta(a\cdot_{V} d) + \beta^{2}(b)\blacktriangleleft((a\diamond_{V} c)\blacktriangleright \beta(d))\\
&=r(\alpha (T(a))\cdot(T(b)\cdot T(d)))\beta^{2}(c) + r(\alpha^{2}(T(d)))\rho([T(a),T(b)])\beta(c)-\ell(\alpha^{2} (T(a)))\ell(\alpha(T(b))) r(T(d))c \\
&\quad +r(\alpha (T(a)\cdot T(d)))\rho(\alpha (T(b)))\beta(c) +\ell(\alpha^{2}(T(b)))r(\alpha (T(d)))\rho(T(a))c=0.
\end{align*}
Similarly, we have
\begin{align*}
&[[a,b]_{V},\beta (c)]_{V}\blacktriangleleft \beta^{2}(d)-\beta^{2} (a)\blacktriangleleft(\beta (b)\blacktriangleleft (c \blacktriangleleft d)) + \beta^{2}(c)\blacktriangleleft(\beta (a)\blacktriangleleft (b\blacktriangleleft d)) \\
&\quad  + \beta ([b,c]_{V})\blacktriangleleft\beta (a \blacktriangleleft d)+ \beta ^{2}(b)\blacktriangleleft([a,c]_{V}\blacktriangleleft\beta(d))\\
&=
\ell([[T(a),T(b)],\alpha (T(c))])\beta^{2}(d)- \ell(\alpha^{2} (T(a)))\ell(\alpha (T(b)))\ell(T(c)) + \ell(\alpha^{2} (T(c)))\ell(\alpha (T(a)))\ell(T(b))d\\
&\quad +\ell(\alpha ([T(b),T(c)])\ell(\alpha (T(a)))\beta(d)+ \ell(\alpha^{2}(T(b)))\ell([T(a),T(c)])\beta(d) = 0.
\qedhere
\end{align*}
\end{proof}
Now, we introduce the following concept of Rota-Baxter operator on a Hom-pre-Malcev algebra which is a particular case of $\mathcal{O}$-operator associated to the regular representation.
 \begin{defn}
 Let $(A, \cdot, \alpha)$ be a Hom-pre-Malcev algebra. A linear map $\mathcal{R}: A \longrightarrow A$ is
called a Rota-Baxter operator of weight zero on $A$ if for all $x, y \in A$,
$$\mathcal{R}\circ\alpha = \alpha\circ \mathcal{R}, \quad \quad
\mathcal{R}(x)\cdot \mathcal{R}(y) = \mathcal{R}\big(\mathcal{R}(x)\cdot y + x\cdot \mathcal{R}(y)\big).$$
\end{defn}
\begin{cor}\label{HPM==>HMD by rota-baxter}
Let $(A, \cdot, \alpha)$ be a Hom-pre-Malcev algebra and $\mathcal{R} : A \to A$ be a Rota-Baxter operator of weight $0$ for $A$. Define new operations on $A$ by
$$x \blacktriangleright  y = x \cdot \mathcal{R}(y), \quad \quad
x \blacktriangleleft y = \mathcal{R}(x) \cdot y.$$
Then $(A,\blacktriangleright, \blacktriangleleft, \alpha)$ is a Hom-M-dendriform algebra.
\end{cor}
\begin{thm}\label{thm:invertibleop}
Let $(A,\cdot, \alpha)$ be a Hom-pre-Malcev algebra. Then there exists a compatible Hom-M-dendriform algebra if and only if there exists an invertible $\mathcal{O}$-operator on $A$ associated to a representation $(V,\ell,r,\beta)$.
\end{thm}
\begin{proof}
 Let $T$  be  an invertible $\mathcal{O}$-operator   of $A$ associated to $(V,\ell,r,\beta)$. Then, using  \eqref{hpm==>dend}, there is a Hom-M-dendriform algebra structure on $T(V)=A$ given for all $a, b\in V$ by
\begin{equation*}
T(a) \blacktriangleright T(b) = T(r(T(b))a), \quad T(a)\blacktriangleleft T(b) = T(\ell(T(a))b).
\end{equation*}
If we put $x=T(a)$ and $y=T(b)$, then we get
\begin{equation*}
x\blacktriangleright y = T(r(y)T^{-1}(x)), \quad x\blacktriangleleft y = T(\ell(x)T^{-1}(y)).
\end{equation*}
This is compatible Hom-M-dendriform algebra structure  on $A$. Indeed,
\begin{align*}
x \blacktriangleleft y+ x\blacktriangleright y &= T(\ell(x)T^{-1}(y))+T(r(y)T^{-1}(x)),\\
& TT^{-1}(x)\cdot TT^{-1}(y) = x\cdot y.
\end{align*}
Conversely, let $(A,\blacktriangleright,\blacktriangleleft,\alpha)$ be a Hom-M-dendriform
algebra  and $(A,\cdot,\alpha)$ be the associated horizontal Hom-pre-Malcev algebra.  Then $(A,L_{\blacktriangleleft},R_{\blacktriangleright},\alpha)$ is a representation of  $(A,\cdot,\alpha)$ and the identity map $id: A \to A$ is an invertible $\mathcal{O}$-operator of $(A,\cdot, \alpha)$ associated
to $(A, L_{\blacktriangleleft}, R_{\blacktriangleright}, \alpha)$.
\end{proof}

\begin{lem}\label{commuting rotabaxter}
Let $\mathcal{R}_1$ and $\mathcal{R}_2$ be two commuting Rota-Baxter operators $($of weight
zero$)$ on a Hom-Malcev algebra $(A, [-, -],\alpha)$. Then $\mathcal{R}_2$ is a Rota-Baxter operator $($of weight zero$)$ on the Hom-pre-Malcev algebra $(A,\cdot,\alpha)$, where for all $x,y \in A$,
$$x \cdot y= [\mathcal{R}_1(x),  y].$$
\end{lem}

\begin{proof}
For any $x,y \in A$,
\begin{align*}
\mathcal{R}_2(x) \cdot \mathcal{R}_2(y) = [\mathcal{R}_1(\mathcal{R}_2(x)), \mathcal{R}_2(y)]
& = \mathcal{R}_2([\mathcal{R}_1(\mathcal{R}_2(x)), y] + [\mathcal{R}_1(x), \mathcal{R}_2(y)]) \\
& = \mathcal{R}_2(\mathcal{R}_2(x) \cdot y + x \cdot \mathcal{R}_2(y)).
\qedhere \end{align*}
%This finishes the proof.
\end{proof}

\begin{cor}
Let $\mathcal{R}_1$ and $\mathcal{R}_2$ be two commuting Rota-Baxter operators
$($of weight zero$)$ on a Hom-Malcev algebra $(A, [-, -],\alpha)$. Then there exists a Hom-M-dendriform algebra structure on $A$ given for all $x,y \in A$ by
\begin{align}\label{ rota-baxter HMD}
    x \blacktriangleright y= [\mathcal{R}_1(x),  \mathcal{R}_2(y)],\quad   x \blacktriangleleft y=[\mathcal{R}_1(\mathcal{R}_2(x)),  y].
\end{align}
\end{cor}

\begin{proof}
By Lemma \ref{commuting rotabaxter}, $\mathcal{R}_2$ is a Rota-Baxter operator of weight zero on $(A,\cdot,\alpha)$, where
$$x \cdot y=[ \mathcal{R}_1(x), y].$$
Then, applying Corollary \ref{HPM==>HMD by rota-baxter}, there exists a Hom-M-dendriform algebraic structure on $A$ given for all $x,y \in A$ by \
$x \blacktriangleright y=x\cdot \mathcal{R}_2(y)= [\mathcal{R}_1(x), \mathcal{R}_2(y)]$ and
$x \blacktriangleleft y= \mathcal{R}_2(x)\cdot y=[\mathcal{R}_1(\mathcal{R}_2(x)), y].$
\end{proof}
\begin{ex}
Consider the $4$-dimensional Hom-Malcev algebra and the Rota-Baxter operators $\mathcal{R}$ given in  {\rm \cite[Example 3.2]{Fattoum}}.
Then, there is a Hom-M-dendriform algebraic structure on $A$ given for all $x,y \in A$ by \
$x \blacktriangleright_{\alpha} y= \alpha([\mathcal{R}(x),  \mathcal{R}(y)])$ and
$x \blacktriangleleft_{\alpha} y= \alpha([\mathcal{R}^2(x),y]),$
that is
$$
\begin{array}{c|cccc}
  \blacktriangleright_{\alpha}  & e_1 & e_2 & e_3 &  e_4 \\[1mm]
  \hline
   \rule{0pt}{1\normalbaselineskip}
  e_1& 0 & \lambda_1 e_3   & 0 & 0 \\
  e_2 & -\lambda_1 e_3  &   0 & 0 & 0\\
  e_3 & 0 & 0  & 0 & 0\\
  e_4 &  0 & 0  & 0 & 0
  \end{array} \quad \quad \quad
\begin{array}{c|cccc}
  \blacktriangleleft_{\alpha}  & e_1 & e_2 & e_3 &  e_4 \\[1mm]
  \hline
   \rule{0pt}{1\normalbaselineskip}
 e_1 & \frac{a_4}{2}e_4  & -\alpha(e_2) & e_3 & -e_4 \\
  e_2 & 0  & 0 &  0 & 0\\
  e_3 & 0 & 0  & 0 & 0\\
  e_4 &  0 & 0  & 0 & 0
  \end{array}.$$
\end{ex}
\begin{ex}
Consider the $5$-dimensional Hom-Malcev algebra and Rota-Baxter operators $\mathcal{R}$ given in  {\rm \cite[Example 3.3]{Fattoum}}.
Then, there is a Hom-M-dendriform algebraic structure on $A$ given for all $x,y \in A$ by\
$x \blacktriangleright_{\alpha} y= \alpha([\mathcal{R}(x),  \mathcal{R}(y)])$,
$x \blacktriangleleft_{\alpha} y=\alpha([\mathcal{R}^2(x),  y]), $
that is
$$
\begin{array}{c|ccccc}
  \blacktriangleright_{\alpha}  & e_1 & e_2 & e_3 &  e_4  & e_5\\[1mm]
  \hline
  \rule{0pt}{1\normalbaselineskip}
  e_1& 0 & 0  & 0 & b e_3 & 0 \\
  e_2 & 0  &   0 & 0 & 0 & 0\\
  e_3 & 0 & 0  & 0 & 0 & 0\\
  e_4 & -b e_3 & 0  & 0 & 0 & 0\\
  e_5 & 0 & 0  & 0 & 0 & 0
  \end{array} \quad \quad \quad
\begin{array}{c|ccccc}
  \blacktriangleleft_{\alpha}  & e_1 & e_2 & e_3 &  e_4 & e_5 \\[1mm]
  \hline
  \rule{0pt}{1\normalbaselineskip}
 e_1 & -a_4e_2  &- a_5e_3 & 0 & e_2 & \frac{- b a_4 }{a_5}e_3 \\
  e_2 & 0  & 0 &  0 & 0 & 0\\
  e_3 & 0 & 0  & 0 & 0 & 0\\
  e_4 &  0 & 0  & 0 & 0 & 0\\
  e_5 &  0 & 0  & 0 & 0 & 0
  \end{array}.$$
\end{ex}
In the sequel, we give the relation between Hessian structure and Hom-M-dendriform algebras.
\begin{defn}
A Hessian structure on a regular Hom-pre-Malcev algebra is a symmetric
nondegenerate satisfying, for all $x, y, z \in A$,
\begin{equation}
\mathcal{B}(\alpha(x), \alpha(y)) = \mathcal{B}(x, y),
\end{equation}
\begin{equation}\label{eq:2-cocycle}
\mathcal{B}(x \cdot y, \alpha(z)) - \mathcal{B}(\alpha(x), y \cdot z) = \mathcal{B}(y \cdot x, \alpha(z)) - \mathcal{B}(\alpha(y), x\cdot z).
\end{equation}
\end{defn}
\begin{prop}
Let $(A, \cdot, \alpha)$ be a regular Hom-pre-Malcev algebra with a  Hessian structure $\mathcal{B}$, and $(A^{*}, ad^{\star},-R_{\cdot}^{\star},(\alpha^{-1})^{*})$ be a representation of $(A, \cdot, \alpha)$. Then there exists a compatible Hom-M-dendriform algebra structure on $(A, \cdot, \alpha)$ given
for all $x,y,z\in A$ by
\begin{equation}
 \mathcal{B}(x\blacktriangleright y,\alpha(z))=\mathcal{B}(\alpha(x),z\cdot y ),\qquad \mathcal{B}(x\blacktriangleleft y,\alpha(z))=-\mathcal{B}(\alpha(y),[x, z]).
\end{equation}
\end{prop}
\begin{proof}
Define the linear map $T : A^{*}\to A$ by $\langle T^{-1}(x), y\rangle = \mathcal{B}(x, y)$. Since $\mathcal{B}$ is $\alpha$-symmetric and using  \eqref{eq:2-cocycle}, we obtain that $T$ is an invertible $\mathcal{O}$-operator on $A$ associated to the representation
$(A^{*}, ad^{\star}, -R_{\cdot}^{\star}, (\alpha^{-1})^{*})$. By Corollary~\ref{thm:invertibleop}, there exists a compatible Hom-M-dendriform algebra
structure given by
$
x\blacktriangleright y = -T(R_{\cdot}^{\star}(y)T^{-1}(x)), \quad x\blacktriangleleft y = T(ad^{\star}(x)T^{-1}(y)).
$
Hence,
\begin{align*}
& \mathcal{B}(x\blacktriangleright y,\alpha(z))=-\big\langle T^{-1}(x\blacktriangleright y,\alpha(z)\big\rangle=-\big\langle R_{\cdot}^{\star}(y)T^{-1}(x),\alpha(z)\big\rangle=\big\langle T^{-1}(x),\alpha(z)\cdot\alpha(y)\big\rangle\\
&\quad  =\mathcal{B}(x, \alpha(z)\cdot\alpha(y))=\mathcal{B}(x, \alpha(z\cdot y))=\mathcal{B}(\alpha(x), \alpha^{2}(z\cdot y))=\mathcal{B}(\alpha(x),z\cdot y),\\
& \mathcal{B}(x\blacktriangleleft y,\alpha(z)) =\big\langle T^{-1}(x\blacktriangleleft y,\alpha(z)\big\rangle=\big\langle ad^{\star}(x)T^{-1}(y),\alpha(z)\big\rangle=-\big\langle T^{-1}(y),[\alpha(x), \alpha(z)]\big\rangle\\
& \quad =-\mathcal{B}(y, [\alpha(x), \alpha(z)] )=-\mathcal{B}(y, \alpha([x,z]))=-\mathcal{B}(\alpha(y), \alpha^{2}([x,z] ))=-\mathcal{B}(\alpha(y),[x, z]).
\qedhere \end{align*}
\end{proof}
%%%%%%%%%%%%%%%%%%%%%%%%%%%%%%%%%%%%%%%%%%%%%%%%%%%%%%%%%%%%%%%%%%%%
\section{Hom-M-dendriform algebras starting from Hom-alternative quadri-algebras}%
%%%%%%%%%%%%%%%%%%%%%%%%%%%%%%%%%%%%%%%%%%%%%%%%%%%%%%%%%%%%%%%%%%%
Hom-M-dendriform algebras are related to Hom-alternative quadri-algebras in the same way Hom-L-dendriform algebras are related to Hom-quadri-algebras (see \cite{ChtiouiMabroukMakhlouf2} for more details).
In  the following, we recall the notion of representation of Hom-pre-alternative algebras  introduced in \cite{Q.Sun}.
\begin{defn}[\cite{Q.Sun}] \label{def:bimodhomprealt}
Let $( A,\prec,\succ,\alpha)$ be a Hom-pre-alternative algebra and $(V,\beta)$ a vector space. Let $L_\succ,R_\succ,L_\prec,,R_\prec:  A\to End(V)$ be linear maps.
Then, $(V ,L_\succ,R_\succ,L_\prec,,R_\prec,\beta)$ ia called a
representation  of $( A\prec,\succ,\alpha)$ if for any $x, y \in A$, and $x\ast y=x\prec y+y\succ x$, $L=L_\prec+L_\succ$ and $R=R_\prec+R_\succ$,
\begin{eqnarray}
 L_\succ(x\ast y+y\ast x)\beta&=&L_\succ(\alpha(x))L_\succ(y)+L_\succ(\alpha(y))L_\succ(x),\label{pabm1}\\
R_\succ(\alpha(y))(L(x)+R(x))&=&L_\succ(\alpha(x))R_\succ(y)+R_\succ(x\succ y)\beta,\\
R_\prec(\alpha(y))L_\succ(x)+R_\prec(\alpha(y))R_\prec(x)&=&L_\succ(\alpha(x))R_\prec(y)+R_\prec(x\ast y)\beta,\\
R_\prec(\alpha(y))R_\succ(x)+R_\succ(\alpha(y))L_\prec(x)&=&L_\prec(\alpha(x))R(y)+R_\succ(x\ast y)\beta,\\
L_\prec(y\prec x)\beta+L_\prec(x\succ y)\beta&=&L_\prec(\alpha(y))L(x)+L_\succ(\alpha(y))L_\succ(x),\\
R_\prec(\alpha(x))L_\succ(y)+L_\succ(y\succ x)\beta&=&L_\succ(y)R_\prec(x)+L_\succ(\alpha(y))L_\succ(x),\\
R_\prec(\alpha(x))R_\succ(y)+R_\succ(\alpha(y))R(x)&=&R_\succ(y\prec x)\beta+R_\succ(x\succ y)\beta,\\
L_\prec(y\succ x)\beta+R_\succ(\alpha(x))L(y)&=&L_\succ(\alpha(y))L_\prec(x)+L_\succ(\alpha(y))R_\succ(y),\\
R_\prec(\alpha(x))R_\prec(y)+R_\prec(\alpha(y))R_\prec(x)&=&R_\prec(x\ast y+y\ast x)\beta,\\
R_\prec(\alpha(y))L_\prec(x)+L_\prec(x\prec y)\beta&=&L_\prec(\alpha(x))(R(y)+L(y)).\label{pabm10}
\end{eqnarray}
\end{defn}
\begin{prop}\label{directsumhomprealt}
A tuple
$(V,L_\succ,R_\succ,L_\prec,R_\prec,\beta)$ is a representation of a Hom-pre-alter\-native algebra
$( A,\prec,\succ,\alpha)$  if and only if the direct sum $( A\oplus V, \ll,\gg,\alpha+\beta)$ is a Hom-pre-alternative algebra, where for any $x,y \in  A, a,b \in V$,
\begin{align*}
  (x+a)\ll (y+b) &= x\prec y+L_\prec(x)b+R_\prec(y)a,\\
   (x+a)\gg (y+b) &= x\succ y+ L_\succ(x)b+R_\succ(x)a, \\
 (\alpha\oplus\beta)(x+a) &= \alpha(x)+\beta(a).
\end{align*}
We denote it by $A \ltimes^{\alpha,\beta}_{L_\succ,R_\succ,L_\prec,R_\prec} V$ or simply $A \ltimes V$.
\end{prop}
\begin{defn}
Let $(V,\mathfrak{l},\mathfrak{r},\beta)$ be a representation of a Hom-alternative algebra $(A,\ast,\alpha)$.
An $\mathcal{O}$-operator of Hom-alternative algebra $(A,\ast,\alpha)$ with respect to the representation $(V,\mathfrak{l},\mathfrak{r},\beta)$ is a linear map $T:V\to A$ such that,
for all $a, b \in V$,
\begin{equation}
\label{O-ophomalternative} \alpha\circ T =  T\circ\beta ~~\text{and}~~T (a)\ast T (b) = T \big(\mathfrak{l}(T (a))b + \mathfrak{r}(T (b))a\big).
 \end{equation}
 \end{defn}
 \begin{rem}
Rota-Baxter operator of weight $0$ on a Hom-alternative algebra $(A,\ast, \alpha)$ is
an $\mathcal{O}$-operator associated to the representation $(A,L, R,\alpha)$, where $L$ and $R$ are the left and right
multiplication operators corresponding to the multiplication $\ast$.
 \end{rem}

Recall  from \cite{myung} that a Hom-algebra $(A, [-, -], \alpha)$  is said to be a Hom-Malcev admissible algebra if, for any
elements $x, y \in A$, the bracket $[-,-] : A \times A \to A$ defined by
$[x, y] = x \ast y - y \ast x$
satisfies the Hom-Malcev identity.

\begin{prop}[\cite{Fattoum}]\label{Hom-pre-altToHom-Pre-Malcev}
    Let $T:V\to A$ be an $\mathcal{O}$-operator of Hom-alternative algebra $(A,\ast,\alpha)$ with respect to the representation $(V,\mathfrak{l},\mathfrak{r},\beta)$. Then $(V,\prec,\succ, \beta)$ be a Hom-pre-alternative algebra, where for all $a,b\in V$,
\begin{equation}
  \label{homalt==>prehomalt} a\succ b= \mathfrak{l}(T (a))b \ \ \text{and}\ \ a\prec b= \mathfrak{r}(T (b))a.
\end{equation}
 Moreover, if $(V,\cdot,\beta)$ is the   multiplicative Hom-pre-Malcev  algebra associated to the  Hom-Malcev admissible algebra $(A,[-,-],\alpha)$ on the representation $(V,\mathfrak{l}-\mathfrak{r},\beta)$, then $a\cdot b=a\succ b-b\prec a$.
\end{prop}
\begin{cor}[\cite{Fattoum}]\label{homalt==>homprealt}
Let $(A, \ast , \alpha)$ be a Hom-alternative algebra and $\mathcal{R}:A\rightarrow A$ be a Rota-Baxter operator
of weight $0$ such that $ \mathcal{R}\alpha =\alpha  \mathcal{R}$. If multiplications
$\prec $ and $\succ $ on $A$ are defined for all $x, y\in A$ by
$x\prec y = x\ast \mathcal{R}(y)$ and $x\succ y = \mathcal{R}(x)\ast y$,
then $(A, \prec , \succ , \alpha)$ is a Hom-pre-alternative algebra.

Moreover, if $(A,\cdot,\alpha)$ be the  multiplicative Hom-pre-Malcev  algebra associated to the  Hom-Malcev admissible algebra $(A,[-,-],\alpha)$, then $x\cdot y=x\succ y-y\prec x$.
\end{cor}
Now, we introduce the Hom version of alternative quadri-algebras.
\begin{defn}\label{def:altquad}
    $\textbf{Hom-alternative quadri-algebra}$ is a 6-tuple $(A, \nwarrow, \swarrow, \nearrow, \searrow, \alpha)$
consisting of a vector  space $A$, four bilinear maps $\nwarrow,
\swarrow, \nearrow, \searrow: A \times A\rightarrow A$ and a linear map $\alpha
: A\rightarrow A$ which is algebra morphism such that the following axioms are satisfied for all $x, y, z\in A:$
    \begin{alignat*}{4}
     \las x,y,z \ras^{r}_{\alpha} + \las y,x,z \ras^{m}_{\alpha} &= 0, &
        \qquad \qquad
       \las x,y,z \ras^{r}_{\alpha} + \las x,z,y \ras^{r}_{\alpha} &= 0,
        \\
        \las x,y,z \ras^{n}_{\alpha} + \las y,x,z \ras^{w}_{\alpha} &= 0, &
        \qquad \qquad
        \las x,y,z \ras^{n}_{\alpha} + \las x,z,y \ras^{ne}_{\alpha} &= 0,
        \\
        \las x,y,z \ras^{ne}_{\alpha} + \las y,x,z \ras^{e}_{\alpha} &= 0, &
        \qquad \qquad
        \las x,y,z \ras^{w}_{\alpha} + \las x,z,y \ras^{sw}_{\alpha} &= 0,
        \\
        \las x,y,z \ras^{sw}_{\alpha} + \las y,x,z \ras^{s}_{\alpha} &= 0, &
        \qquad \qquad
        \las x,y,z \ras^{m}_{\alpha} + \las x,z,y \ras^{\ell}_{\alpha} &= 0,
        \\
        \las x,y,z \ras^{\ell}_{\alpha} + \las y,x,z \ras^{\ell}_{\alpha} &= 0, &
    \end{alignat*}
where
    \begin{alignat*}{4}
        \las x,y,z \ras^{r}_{\alpha} &= ( x \nwarrow y ) \nwarrow \alpha(z) - \alpha(x) \nwarrow ( y \ast z ) && \quad \text{\rm (right $\alpha$-associator)} \\
        \las x,y,z \ras^{\ell}_{\alpha} &= ( x \ast y ) \searrow \alpha(z) - \alpha(x) \searrow ( y \searrow z ) && \quad \text{\rm (left $\alpha$-associator)} \\
        \las x,y,z \ras^{m}_{\alpha} &= ( x \searrow y ) \nwarrow \alpha(z) - \alpha(x) \searrow ( y \nwarrow z )  && \quad \text{\rm (middle $\alpha$-associator)}\\
        \las x,y,z \ras^{n}_{\alpha} &= ( x \nearrow y ) \nwarrow \alpha(z) - \alpha(x) \nearrow ( y \prec z ) && \quad \text{\rm (north $\alpha$-associator)} \\
        \las x,y,z \ras^{w}_{\alpha} &= ( x \swarrow y ) \nwarrow \alpha(z) - \alpha(x) \swarrow ( y \wedge z) && \quad \text{\rm (west $\alpha$-associator)} \\
        \las x,y,z \ras^{s}_{\alpha} &= ( x \succ y ) \swarrow \alpha(z) - \alpha(x) \searrow ( y \swarrow z ) && \quad \text{\rm (south $\alpha$-associator)} \\
        \las x,y,z \ras^{e}_{\alpha} &= ( x \vee y ) \nearrow \alpha(z) - \alpha(x) \searrow ( y \nearrow z ) && \quad \text{\rm (east $\alpha$-associator)} \\
        \las x,y,z \ras^{ne}_{\alpha} &= ( x \wedge y ) \nearrow \alpha(z) - \alpha(x) \nearrow ( y \succ z ) && \quad \text{\rm (north-east $\alpha$-associator)} \\
        \las x,y,z \ras^{sw}_{\alpha} &= ( x \prec y ) \swarrow \alpha(z) - \alpha(x) \swarrow ( y \vee z ) && \quad \text{\rm (south-west $\alpha$-associator)}
   \end{alignat*}
   \begin{alignat*}{3}
   x \succ y&= x \nearrow y+ x \searrow y ,  & x \prec y= x \nwarrow y + x \swarrow y, \\
   x \vee y&= x \searrow y+ x \swarrow y,   & x \wedge y= x \nearrow y + x \nwarrow y, \\
%     \end{alignat*}
%    \begin{align*}
        x \ast y   &= x \succ y + x \prec y &= x \searrow y + x \nearrow y + x \nwarrow y + x \swarrow y.
    \end{alignat*}
\end{defn}

\begin{lem}
    Let $(A,\nearrow, \searrow, \swarrow, \nwarrow,\alpha)$ be some Hom-alternative quadri-algebra. Then
    $(A,\prec,\succ,\alpha)$ and $(A,\vee,\wedge,\alpha)$ are Hom-pre-alternative algebras
    (called respectively horizontal and vertical Hom-pre-alternative structures associated to $A$),
    and $(A,\ast,\alpha)$ is a Hom-alternative algebra.
\end{lem}
A morphism $f:(A, \nwarrow, \swarrow, \nearrow, \searrow, \alpha)\rightarrow (A', \nwarrow', \swarrow', \nearrow', \searrow',
\alpha' )$ of Hom-alternative quadri-algebras is a linear map
$f:A\rightarrow A'$ satisfying $f(x\nearrow y)=f(x)\nearrow' f(y),
f(x\searrow y)=f(x)\searrow' f(y), f(x\nwarrow y)=f(x)\nwarrow'
f(y)$ and $f(x\swarrow y)=f(x)\swarrow ' f(y)$, for all $x, y\in A$,
as well as $f \alpha =\alpha ' f$.

\begin{prop} \label{quad}
Let $(A, \nwarrow, \swarrow, \nearrow, \searrow )$ be an
alternative quadri-algebra and $\alpha:A\rightarrow A$  an
alternative quadri-algebra endomorphism. Define $\searrow _{\alpha},
\nearrow _{\alpha}, \swarrow _{\alpha},
\nwarrow _{\alpha}: A\times A \rightarrow A$ by
\begin{eqnarray*}
&&x\searrow _{\alpha}y=\alpha (x)\searrow \alpha (y),
\quad \quad x\nearrow _{\alpha}y=\alpha (x)\nearrow \alpha
(y),\\
&&x\swarrow _{\alpha}y=\alpha (x)\swarrow \alpha (y),
\quad \quad x\nwarrow _{\alpha}y=\alpha (x)\nwarrow \alpha
(y),
\end{eqnarray*}
for all $x, y\in A$.
Then $A_{\alpha}:=(A, \nwarrow _{\alpha},  \swarrow
_{\alpha}, \nearrow _{\alpha},
\searrow
_{\alpha},  \alpha )$ is a multiplicative Hom-alternative quadri-algebra, called the Yau twist of $A$. Moreover,
assume that $(A', \nwarrow', \swarrow', \nearrow', \searrow' )$ is
another alternative quadri-algebra and $\alpha ': A'\rightarrow A'$ is
an alternative quadri-algebra morphism and $f:A\rightarrow A'$
is an alternative quadri-algebra morphism satisfying $f \alpha =\alpha
' f$. Then $f:A_{\alpha}\rightarrow A'_{\alpha '}$ is a
Hom-alternative quadri-algebra endomorphism.
\end{prop}

 \begin{proof}
 We define the following operations $x\succ _{\alpha}y:=x\nearrow _{\alpha}y+
x\searrow _{\alpha}y$,
$x\prec _{\alpha}y:=x\nwarrow _{\alpha}y+ x\swarrow _{\alpha}y$,
$x\vee _{\alpha}y:=x\searrow _{\alpha}y+ x\swarrow _{\alpha }y$,
$x\wedge _{\alpha}y:=x\nearrow _{\alpha}y+ x\nwarrow _{\alpha }y$
and $x\ast _{\alpha}y:=x\searrow _{\alpha}y+ x\nearrow _{\alpha}y+
x\swarrow _{\alpha}y+ x\nwarrow _{\alpha}y$, for all $x, y\in A$.
It is easy to get $x\succ _{\alpha}
y= \alpha (x)\succ \alpha (y), x\prec _{\alpha } y= \alpha
(x)\prec \alpha (y), x\vee _{\alpha} y= \alpha (x)\vee
\alpha (y), x\wedge _{\alpha} y= \alpha (x)\wedge \alpha
(y)$ and $x\ast _{\alpha} y= \alpha (x)\ast \alpha (y)$
for all $x, y \in A$. Since $\alpha$ is a quadri-alternative algebra endomorphism,
for all $x, y, z \in A$,
\begin{align*}
\las x,y,z \ras^{ne}_{\alpha} + \las y,x,z \ras^{e}_{\alpha}
& =(x \wedge_{\alpha} y) \nearrow_{\alpha} \alpha(z)-\alpha(x) \nearrow_{\alpha}(y \succ_{\alpha} z)
 \\
& \quad
+ (y \vee_{\alpha} x) \nearrow_{\alpha} \alpha(z) -\alpha(y) \searrow_{\alpha}(x \nearrow_{\alpha} z) \\
       &=(\alpha^2(x) \wedge \alpha^2 (y)) \nearrow \alpha^2(z)-\alpha^2(x) \nearrow(\alpha^2(y) \succ \alpha^2(z)) \\
& \quad + (\alpha^2(y) \vee \alpha^2(x)) \nearrow \alpha^2(z) -\alpha^2(y) \searrow (\alpha^2(x) \nearrow\alpha^2(z)) \\
& = \alpha^2 \las x,y,z \ras^{ne}+ \alpha^2 \las y,x,z \ras^{e} =0.
\qedhere
\end{align*}
 \end{proof}

\begin{rem}
 Let $(A, \nwarrow, \swarrow, \nearrow, \searrow, \alpha)$ be a  Hom-alternative quadri-algebra and $\gamma: A \to A$ be a  Hom-alternative quadri-algebra morphism. Define new multiplications on $A$ by
 \begin{align*}
    & x \nearrow' y= \gamma(x) \nearrow \gamma(y), \quad
    x \searrow' y= \gamma(x) \searrow \gamma(y),\\
   & x \nwarrow' y= \gamma(x) \nwarrow \gamma(y), \quad
   x \swarrow' y= \gamma(x) \swarrow \gamma(y).
 \end{align*}
 Then, $(A',\nearrow',\searrow',\swarrow',\nwarrow',\alpha \circ \gamma)$ is a multiplicative Hom-alternative quadri-algebra.
\end{rem}

\begin{defn}
Let $(A,\prec,\succ,\alpha)$ be a Hom-pre-alternative algebra. A linear map $T: V \to A$ is called an $\mathcal{O}$-operator of  $(A,\prec,\succ,\alpha)$ associated to a representation   $(V,L_{\prec},R_{\prec},L_{\succ},R_{\succ},\beta)$  if for all $a, b \in V$,
 \begin{align}\label{Oophomprealt}
\begin{split}
T \circ\beta&=\alpha\circ T,\\
 T(a) \succ T(b)&=T\big(L_{\succ}(T(a))b+R_{\succ}(T(b))a\big),\\
  T(a) \prec T(b)&=T\big(L_{\prec}(T(a))b+R_{\prec}(T(b))a\big).
\end{split}
\end{align}
\end{defn}
We have the following results.
\begin{prop}
Let  $(V,L_{\prec},R_{\prec},L_{\succ},R_{\succ},\beta)$ be a representation of a Hom-pre-alternative algebra  $(A,\prec,\succ,\alpha)$ and $(A,\ast, \alpha)$ be the associated Hom-alternative algebra.
If $T$ is  an $\mathcal{O}$-operator of  $(A,\prec,\succ,\alpha)$ associated to $(V,L_{\prec},R_{\prec},L_{\succ},R_{\succ},\beta)$,  then $T$ is an $\mathcal{O}$-operator of $(A,\circ,\alpha)$ associated to $(V,L_{\prec}+L_{\succ}, R_{\prec}+R_{\succ},\beta)$.
\end{prop}

\begin{prop}\label{last prop}
If $(A,\prec,\succ,\alpha)$ is a Hom-pre-alternative algebra and $T$ is an $\mathcal{O}$-operator of  $(A,\prec,\succ,\alpha)$ associated to a representation $(V,L_{\prec},R_{\prec},L_{\succ},R_{\succ},\beta)$, then $(V,\searrow,\nearrow,\swarrow,\nwarrow,\beta)$ is a Hom-alternative quadri-algebra with products defined, for all $a,b \in V$, by
\begin{equation}
\begin{array}{ll}
 a\searrow b =L_{\succ}(T(a))b, & a \nearrow b=R_{\succ}(T(b))a, \\
 a \swarrow b = L_{\prec}(T(a))b, & a \nwarrow b=R_{\prec}(T(b))a.
\end{array}
\end{equation}
\end{prop}
\begin{proof}
Set $L=L_{\prec}+L_{\succ}$ and $R=R_{\prec}+R_{\succ}$.  For any $a,b,c \in V$,
\begin{align*}
    &(a \ast b)\se \beta(c)-\beta(a) \se (b \se c) \\
    &\quad =(L(T(a))b+R(T(b))a)\se \beta(c)-\beta(a)\se (L_{\succ}(T(b))c) \\
    &\quad =L_{\succ}(T(L(T(a))b+R(T(b))a))\beta(c)-L_{\succ}(\alpha(T(a)))L_{\succ}(T(b))c \\
    &\quad =L_{\succ}(T(a)\ast T(b)) \beta(c)-L_{\succ}(\alpha(T(a)))L_{\succ}(T(b))c=0.
\\
& (a \sw b) \nw \beta(c)-\phi(a) \sw (b \nw c+ b \ne c) \\
& \quad =R_{\prec}(\alpha(T(c)))L_{\prec}(T(a))b-L_{\prec}(\alpha(T(a)))(R_{\prec}(T(c))b+R_{\succ}(T(c))b)\\
& \quad =R_{\prec}(\alpha(T(c)))L_{\prec}(T(a))b-L_{\prec}(\alpha(T(a)))R(T(c))b,
\\
& (a \nw c+a \sw c ) \sw \beta(b)-\beta(a) \sw (c \se b+ c \sw b)\\
& \quad= L_{\prec}(T(R_{\prec}(T(c))a+L_{\prec}(T(a))c))\beta(b)- L_{\prec}(\alpha(T(a)))(L_{\succ}(T(c))b+L_{\prec}(T(c))b) \\
& \quad=L_{\prec}(T(a) \prec T(c))\beta(b)-L_{\prec}(\alpha(T(a)))(L(T(c))b.
\end{align*}
This means that
\begin{align*}
  \las a,b,c \ras^{w}_{\beta}+\las a,c,b \ras^{sw}_{\beta} =&
  R_{\prec}(\alpha(T(c)))L_{\prec}(T(a))b-L_{\prec}(\alpha(T(a)))R(T(c))b\\
 &+L_{\prec}(T(a) \prec T(c))\beta(b)-L_{\prec}(\alpha(T(a)))L(T(c))b=0,
\end{align*}
since $(V,L_{\prec},R_{\prec},L_{\succ},R_{\succ},\beta)$ is a representation of $(A,\prec,\succ,\alpha)$.
The rest of identities can be proved using  analogous  computations.\end{proof}

Rota-Baxter operators $\mathcal{R}:
A\rightarrow A$ are a special case of $\mathcal{O}$-operators on a Hom-pre-alternative
algebra $(A,\prec,\succ,\alpha)$
satisfying  the following conditions, for all $x, y\in A$,
\begin{eqnarray} \mathcal{R}\circ\alpha &=&\alpha\circ  \mathcal{R},\\ \mathcal{R}(x)\succ \mathcal{R}(y)&=&\mathcal{R}(x\succ \mathcal{R}(y)+ \mathcal{R}(x)\succ y),  \label{baxter1}\\
 \mathcal{R}(x)\prec \mathcal{R}(y)&=&\mathcal{R}(x\prec \mathcal{R}(y)+ \mathcal{R}(x)\prec y).  \label{baxter2}
\end{eqnarray}

Analogously to Proposition \ref{last prop}, we have the following construction.
\begin{cor} \label{operation}
Let $(A, \prec , \succ , \alpha)$ be a Hom-pre-alternative
algebra and $\mathcal{R}: A\rightarrow A$ be a Rota-Baxter operator of weight $0$ for $A$.
Then $(A,\nearrow, \searrow, \swarrow, \nwarrow,\alpha)$ is a Hom-alternative quadri-algebra where the operations $\nearrow, \searrow, \swarrow, \nwarrow$ are given, for $x,\ y\in A,$ by
    \begin{eqnarray*}
        x \nearrow y = x \succ \mathcal{R}(y),
        \qquad \qquad
        x \searrow y = \mathcal{R}(x) \succ y,
        \\
        x \swarrow y = \mathcal{R}(x) \prec y,
        \qquad \qquad
        x \nwarrow y = x \prec \mathcal{R}(y).
\end{eqnarray*}
\end{cor}
\begin{prop}\label{pairRB}
Let $(A,\ast,\alpha)$ be a Hom-alternative algebra and $\mathcal{R}_1,\mathcal{R}_2 $ be two commuting Rota-Baxter operators on $A$ such that $\mathcal{R}_1 \alpha =\alpha  \mathcal{R}_1$ and $\mathcal{R}_2 \alpha =\alpha  \mathcal{R}_2$. Then $\mathcal{R}_2$ is a Rota-Baxter operator on the Hom-pre-alternative algebra $(A,\prec,\succ,\alpha)$. where
$x \prec_R y=x\ast \mathcal{R}_1(y)$ and $x \succ_R y =\mathcal{R}_1(x)\ast y$. Moreover, there exists
a Hom-alternative quadri-algebra structure on the underlying vector
space $(A,\ast, \alpha)$, with operations defined by
\begin{eqnarray*}
&& x\searrow y=\mathcal{R}_2(x)\succ y=\mathcal{R}_1(\mathcal{R}_2(x))\ast y,\\
&& x\nearrow y=x\succ \mathcal{R}_2(y)=\mathcal{R}_1(x)\ast \mathcal{R}_2(y),\\
&& x\swarrow y=\mathcal{R}_2(x)\prec y=\mathcal{R}_2(x)\ast \mathcal{R}_1(y),\\
&& x\nwarrow y=x\prec \mathcal{R}_2(y)=x\ast \mathcal{R}_1(\mathcal{R}_2(y)).
\end{eqnarray*}
\end{prop}
\begin{proof}
For all
$x, y\in A$,
\begin{align*}
\mathcal{R}_2 (x)\succ \mathcal{R}_2 (y)&= \mathcal{R}_1(\mathcal{R}_2 (x))\ast \mathcal{R}_2(y)=\mathcal{R}_2 (\mathcal{R}_1(x))\ast \mathcal{R}_2(y)\\
&=\mathcal{R}_2(\mathcal{R}_1(x)\ast\mathcal{R}_2(y)+\mathcal{R}_2(\mathcal{R}_1(x))\ast y)=\mathcal{R}_2 (\mathcal{R}_1(x)\ast\mathcal{R}_2(y)+\mathcal{R}_1(\mathcal{R}_2(x))\ast y)\\
&= \mathcal{R}_2(x\succ \mathcal{R}_2(y)+ \mathcal{R}_2(x)\succ y),\\
\mathcal{R}_2 (x)\prec \mathcal{R}_2 (y)&= \mathcal{R}_2 (x)\ast \mathcal{R}_1(\mathcal{R}_2(y))=\mathcal{R}_2(x)\ast \mathcal{R}_2(\mathcal{R}_1(y))\\
&=\mathcal{R}_2(x\ast\mathcal{R}_2(\mathcal{R}_1(y))+\mathcal{R}_2(x)\ast\mathcal{R}_1(y))=\mathcal{R}_2 (x\ast\mathcal{R}_1(\mathcal{R}_2(y))+\mathcal{R}_2(x)\ast\mathcal{R}_1(y))\\
&= \mathcal{R}_2(x\prec \mathcal{R}_2(y)+ \mathcal{R}_2(x)\prec y).
\end{align*}
On the other hand,
we use the construction in Corollary \ref{operation} with the Rota-Baxter operator $\mathcal{R}_2$  of weight $0$  and the Hom-pre-alternative  algebra $(A, \prec, \succ,
\alpha)$ .
\end{proof}
\begin{thm} Assume hypothesis of  Corollary \ref{homalt==>homprealt} and
let  $(V,L_{\prec},R_{\prec},L_{\succ},R_{\succ},\beta)$  be a representation and $T$ be an $\mathcal{O}$-operator of  $(A,\prec,\succ,\alpha)$ associated to  $(V,L_{\prec},R_{\prec},L_{\succ},R_{\succ},\beta)$. Then   $(V,L_{\succ}-R_{\prec},R_{\succ}-L_{\prec},\beta)$ is a representation of the Hom-pre-Malcev algebra $(A,\cdot,\alpha)$,   and $T$ is an $\mathcal{O}$-operator of $(A,\cdot,\alpha)$ with respect to $(V,L_{\succ}-R_{\prec},R_{\succ}-L_{\prec},\beta)$.  Moreover, if $(V,\blacktriangleright,\blacktriangleleft,\beta)$ is the  Hom-M-dendriform  algebra associated to the  Hom-pre-Malcev algebra $(A,\cdot,\alpha)$ on the representation $(V,L_{\succ}-R_{\prec},R_{\succ}-L_{\prec},\beta)$, then
 \[     a\blacktriangleright b = a \nearrow b - b \swarrow a,
 \qquad
        a  \blacktriangleleft  b = a \searrow b - b \nwarrow a.
    \]
\end{thm}
\begin{proof}
By Proposition \ref{directsumhomprealt}, $A\ltimes V$ is a Hom-pre-alternative algebra. Consider its associated Hom-pre-Malcev algebra $(A \oplus V, \bullet, \alpha+\beta)$,
\begin{align*}
& (x+a)\bullet (y+b) =(x+a)\gg (y+b) - (y+b)\ll (x+a) \\
&= x \succ y+ L_{\succ}(x)b+ R_{\succ}(y)a - y \prec x - L_{\prec}(y)a - R_{\prec}(x)b \\
&= x\cdot y + (L_{\succ}-R_{\prec})(x)b + ( R_{\succ}-L_{\prec})(y)a.
\end{align*}
By Proposition \ref{semidirectproduct hompreMalcev}, $(V,L_{\succ}-R_{\prec},R_{\succ}-L_{\prec},\beta)$ is a representation of Hom-pre-Malcev algebra $(A,\cdot,\alpha)$. Also, $T$ is $\mathcal{O}$-operator of $(A,\cdot,\alpha)$ with respect to $(V,L_{\succ}-R_{\prec},R_{\succ}-L_{\prec},\beta)$:
\begin{align*}
T(a)\cdot T(b)=&T(a)\succ T(b)-T(b)\prec T(a)\\
=&T \big(L_{\succ}(T (a))b + R_{\succ}(T (b))a\big)-T \big(L_{\prec}(T (b))a + R_{\prec}(T (a))b\big)\\
=&T \big((L_{\succ}-R_{\prec})(T (a))b + (R_{\succ}-L_{\prec})(T (b))a\big).
\end{align*}
Moreover, by \eqref{hpm==>dend} and \eqref{Oophomprealt},
\begin{align*}
&a\blacktriangleright b=(R_{\succ}-L_{\prec})(T(b))a=R_{\succ}(T(b))a-L_{\prec}(T(b))a=a\nearrow b - b \swarrow a,\\
&a \blacktriangleleft b =(L_{\succ}-R_{\prec})(T(a))b= L_{\succ}(T(a))b-R_{\prec}(T(a))b=a \searrow b - b \nwarrow a.
\qedhere \end{align*}
\end{proof}
\begin{cor}
Assume hypothesis of  Corollary \ref{homalt==>homprealt}. Let $\mathcal{R}$ be a Rota-Baxter operator of  $(A,\prec,\succ,\alpha)$ and $(A,\blacktriangleright,\blacktriangleleft,\alpha)$ be the  Hom-M-dendriform  algebra associated to the  Hom-pre-Malcev algebra $(A,\cdot,\alpha)$ given in Corollary \ref{HPM==>HMD by rota-baxter}. Then, the operations
 \[     x\blacktriangleright y = x \nearrow y - y \swarrow x,
 \qquad
        x  \blacktriangleleft  y = x \searrow y - y \nwarrow x.
    \]
    define a Hom-M-dendriform structure in $A$ with respect the twisting map $\alpha$.
\end{cor}

\noindent
Summarizing the above study in this section, we have the following commutative diagram:
{\small
\begin{equation*}\label{diagramhommalcev}
    \begin{split}
\resizebox{14cm}{!}{\xymatrix{
\ar[rr] \mbox{\bf Hom-alt quadri-alg $(A,\nwarrow, \swarrow, \nearrow, \searrow,\alpha)$ }\ar[d]_{\mbox{$\begin{array}{l} \succ=\nearrow +\searrow \\ \prec=\nwarrow+ \swarrow\end{array}$}}\ar[rr]^{\mbox{ \quad$\blacktriangleright= \nearrow -\swarrow$}}_{\mbox{ \quad$\blacktriangleleft=\searrow -\nwarrow$}}
                && \mbox{\bf  Hom-M-dendri alg  $(A,\blacktriangleright,\blacktriangleleft,\alpha)$ }\ar[d]_{\mbox{$\cdot=\blacktriangleright+\blacktriangleleft$}}\\
\ar[rr] \mbox{\bf Hom-pre-alt alg $(A,\prec,\succ,\alpha)$}\ar@<-1ex>[u]_{\mbox{R-B }}\ar[d]_{\mbox{ $\star=\prec+\succ$}}\ar[rr]^{\mbox{\quad\quad $\cdot=\prec-\succ$\quad\quad  }}
                && \mbox{\bf Hom-pre-Malcev alg  $(A,\cdot,\alpha)$}\ar@<-1ex>[u]_{\mbox{R-B}}\ar[d]_{\mbox{Commutator}}\\
          \ar[rr] \mbox{\bf Hom-alt alg $(A,\ast,\alpha)$}\ar@<-1ex>[u]_{ \mbox{R-B}}\ar[rr]^{\mbox{Commutator}}
                && \mbox{\bf Hom-Malcev alg  $(A,[-,-],\alpha)$}\ar@<-1ex>[u]_{\mbox{R-B}}}
}
 \end{split}
\end{equation*}
}

%%%%%%%%%%%%%%%%%%%%%


\begin{thebibliography}{999}
%%%%%%%%%%%%%%%%%%%%%
%\bibitem{AbdaouiMabroukMakhlouf}
%Abdaoui, E., Mabrouk, S., Makhlouf, A.: Rota-Baxter Operators on Pre-Lie Superalgebras, Bulletin of the Malaysian Math. Sci. Soc., 1-40 (2017)
\bibitem{Aguiar00}
Aguiar, M:
Pre-Poisson algebras, Lett. Math. Phys. \textbf{54}, 263-277 (2000)

\bibitem{Aguiar}
Aguiar, M.: Infinitesimal bialgebras, pre-Lie and dendriform algebras, In: Hopf Algebras,
Lecture Notes in Pure and Applied Mathematics, Vol. 237, Marcel Dekker, (2004), pp. 1-33.

\bibitem{AL}
Aguiar, M., Loday, J.-L.: Quadri-algebras, J. Pure Applied Algebra, \textbf{191}, 205-221 (2004)

\bibitem{AmmarEjbehiMakhlouf:homdeformation}
Ammar, F., Ejbehi, Z., Makhlouf, A.: Cohomology and deformations of Hom-algebras,  J. Lie Theory, \textbf{21}(4), 813-836 (2011)

%\bibitem{AttanLaraiedh:2020ConstrBihomalternBih%omJordan}
%Attan, S, Laraiedh, I: Construtions and %bimodules of BiHom-alternative and BiHom-Jordan %algebras, 	arXiv:2008.07020 [math.RA] (2020)

\bibitem{Atkinson}
Atkinson, F. V.: Some aspects of Baxters functional equation, J. Math. Anal. Appl. \textbf{7}, 1-30 (1963)

\bibitem {Bai2}
Bai, C. M.: A unified algebraic approach to classical Yang-Baxter equation,
J. Phys. A: Math. Theor. \textbf{40}(36), 11073-11082 (2007)

%\bibitem{Bakayoko:LaplacehomLiequasibialg}
%Bakayoko, I.: Laplacian of Hom-Lie quasi-bialgebras, International Journal of Algebra, \textbf{8}(15), 713-727 (2014)

%\bibitem{Bakayoko:LmodcomodhomLiequasibialg}
%Bakayoko, I.: $L$-modules, $L$-comodules and Hom-Lie quasi-bialgebras, African Diaspora Journal of Mathematics, \textbf{17}, 49-64 (2014)

%\bibitem{BakBan:bimodrotbaxt}
%Bakayoko, I., Banagoura, M.: Bimodules and %Rota-Baxter Relations. J. Appl. Mech. Eng. %\textbf{4}(5) (2015)

%\bibitem{BakyokoSilvestrov:MultiplicnHomLiecolo%ralg}
%Bakayoko, I., Silvestrov, S.: Multiplicative %$n$-Hom-Lie color algebras,
%In: Silvestrov, S., Malyarenko, A., %Ran\u{c}i\'{c}, M. (Eds.), Algebraic Structures %and Applications, Springer Proceedings in %Mathematics and Statistics \textbf{317}, Ch. 7, %159-187, Springer (2020). %(arXiv:1912.10216[math.QA] (2019))

%\bibitem{BakyokoSilvestrov:HomleftsymHomdendicolorYauTwi}
%Bakayoko, I., Silvestrov, S.,
%Hom-left-symmetric color dialgebras, Hom-tridendriform color algebras and Yau’s twisting generalizations, Afrika Matematika \textbf{32}, 941–958 (2021). (arXiv:1912.01441[math.RA])
% \url{https://doi.org/10.1007/s13370-021-00871-z}

\bibitem{baxter}
Baxter, G.: An analytic problem whose solution follows from a simple algebraic identity. Pacific J. Math.
\textbf{10}, 731-742 (1960)

\bibitem{BenMakh:Hombiliform}
 Benayadi, S.,  Makhlouf, A.: Hom-Lie algebras with symmetric invariant nondegenerate bilinear forms, J. Geom. Phys. \textbf{76}, 38-60 (2014)

%\bibitem{BenAbdeljElhamdKaygorMakhl201920GenDernBiHomLiealg}
%Ben Abdeljelil, A., Elhamdadi, M., Kaygorodov, I., Makhlouf, A.: Generalized Derivations of $n$-BiHom-Lie algebras, In: Silvestrov, S., Malyarenko, A., Ran\u{c}i\'{c}, M. (Eds.), Algebraic Structures and Applications,  Springer Proceedings in Mathematics and Statistics \textbf{317}, Ch. 4, 81-97, Springer (2020). (arXiv:1901.09750[math.RA] (2019))
\bibitem{CaiSheng}
 Cai, L.,  Sheng, Y.: Purely Hom-Lie bialgebras. Sci. China Math. \textbf{61}(9), 1553-1566 (2018)
%\bibitem{CaenGoyv:MonHomHopf}
%Caenepeel, S., Goyvaerts, I.: Monoidal Hom-Hopf Algebras,  Comm. Algebra \textbf{39}(6), 2216-2240 (2011)

\bibitem{Cartier}
 Cartier, P.: On the structure of free Baxter algebras, Adv. Math. \textbf{9}, 253-265 (1972)

\bibitem{CHENWANGZHANG1}
 Chen, Y., Wang Z., Zhang, L.: Quasitriangular Hom-Lie bialgebras, J. Lie Theory \textbf{22},
1075-1089 (2012)

\bibitem{ChenZhang1}
Chen, Y., Zhang, L.: Hom-$\mathcal{O}$-operators and Hom-Yang-Baxter equations, Adv. Math. Phys. Art.ID 823756, 11 pp (2015)

\bibitem{CHENWANGZHANG2}
Chen, Y., Zheng, H., Zhang, L.: Double Hom-Associative Algebra and Double Hom–Lie Bialgebra, Adv. Appl. Clifford Algebras \textbf{30}, 8 (2020)

\bibitem{ChtiouiMabroukMakhlouf1}
Chtioui, T., Mabrouk, S., Makhlouf, A.:
BiHom-alternative, BiHom-Malcev and BiHom-Jordan algebras, Rocky Mountain J. Math. \textbf{50}(1), 69-90 (2020) %https://doi.org/10.1216/rmj.2020.50.69

\bibitem{ChtiouiMabroukMakhlouf2}
Chtioui, T., Mabrouk, S., Makhlouf, A.:
BiHom-pre-alternative algebras and BiHom-alternative quadri-algebras,  Bull. Math. Soc. Sci. Math. Roumanie. \textbf{63 (111)}(1), 3-21 (2020)

%\bibitem{DassoundoSilvestrov2021:NearlyHomass}
%Dassoundo, M. L., Silvestrov, S.: Nearly associative and nearly Hom-associative algebras and bialgebras, arXiv:2101.12377 [math.RA] (2021)

\bibitem{EbrahimiFard02}
Ebrahimi-Fard, K.:
Loday-type algebras and the Rota-Baxter relation, Lett. Math. Phys. \textbf{61}, 139-147 (2002)

\bibitem{EbrahimiFardGuo05}
Ebrahimi-Fard, K., Guo, L.:
On products and duality of binary, quadratic, regular operads,
J. Pure Appl. Algebra,
\textbf{200}(3), 293--317 (2005)

\bibitem{EbrahimiFardGuo08}
Ebrahimi-Fard, K., Guo, L.:
Rota-Baxter algebras and dendriform algebras.
J. Pure Appl. Algebra, \textbf{212}(2), 320-339 (2008)

\bibitem{EbrahimiFardGuo09}
Ebrahimi-Fard, K.,  Guo, L., Kreimer, D.:  Integrable renormalization II: the general
case, Ann. Henri Poincar{\'e}, \textbf{6}, 369-395 (2005)

\bibitem{Guo1}
Guo, L., Keigher, W.: Baxter algebras and shuffle products, Adv. Math. \textbf{150}
117-149 (2000)

\bibitem{Guo0}
 Guo, L.: What is a Rota-Baxter algebra,
 Notices. Amer. Math. Soc. \textbf{56}
1436-1437 (2009)

\bibitem{Guo2}
Guo, L., Zhang, B.:
Renormalization of multiple zeta values, J. Algebra, \textbf{319},
3770-3809 (2008)

\bibitem{Fattoum}
 Harrathi, F.,  Mabrouk, S.,  Ncib,  O.,  Silvestrov, S.: Kupershmidt operators on Hom-Malcev algebras and their deformation,  arXiv:2205.03762v1 [math.RA], 31 pp (2022)

\bibitem{gt}
G\"{u}rsey, F., Tze, C.-H.: On the role of division, Jordan and related algebras in particle physics, World Scientific, (1996)

\bibitem{HartwigLarSil:defLiesigmaderiv}
Hartwig, J. T., Larsson, D., Silvestrov, S. D.:
Deformations of Lie algebras using $\sigma$-derivations, J. Algebra, \textbf{295},  314-361 (2006)
(Preprint in Mathematical Sciences 2003:32, LUTFMA-5036-2003, Centre for Mathematical Sciences, Department of Mathematics, Lund Institute of Technology, 52 pp. (2003))

%\bibitem{HassanzadehShapiroSutlu:CyclichomolHomasal}
%Hassanzadeh, M., Shapiro, I., S{\"u}tl{\"u}, S.: Cyclic homology for Hom-associative algebras, J. Geom. Phys. \textbf{98}, 40-56 (2015)

%\bibitem{HounkonnouDassoundo:centersymalgbialg}
%Hounkonnou, M. N., Dassoundo, M. L.: Center-symmetric Algebras and Bialgebras: Relevant Properties and Consequences. In: Kielanowski P., Ali S., Bieliavsky %P., Odzijewicz A., Schlichenmaier M., Voronov T. (eds) Geometric Methods in Physics. Trends in Mathematics. 2016, pp. 281-293. Birkh{\"a}user, Cham (2016)

%\bibitem{HounkonnouHoundedjiSilvestrov:DoubleconstrbiHomFrobalg}
%Hounkonnou, M. N., Houndedji, G. D., Silvestrov, S.: Double constructions of biHom-Frobenius algebras, arXiv:2008.06645 [math.QA] (2020)

%\bibitem{HounkonnouDassoundo:homcensymalgbialg}
%Hounkonnou, M. N., Dassoundo, M. L.: Hom-center-symmetric algebras and bialgebras.  arXiv:1801.06539.

%\bibitem{kms:narygenBiHomLieBiHomassalgebras2020}
%Kitouni, A., Makhlouf, A., Silvestrov, S.: On $n$-ary generalization of BiHom-Lie algebras and BiHom-associative algebras, In: Silvestrov, S., Malyarenko, %A., RancicRan\u{c}i\'{c}, M. (Eds.), Algebraic Structures and Applications, Springer Proceedings in Mathematics and Statistics \textbf{317}, Springer, Ch. %5, 99-126 (2020)

%\bibitem{Laraiedh1:2021:BimodmtchdprsBihomprepois}
%Laraiedh, I: Bimodules and matched pairs of noncommutative BiHom-(pre)-Poisson algebras, arXiv:2102.11364 [math.RA] (2021)

\bibitem{kerdman}
Kerdman, F. S.: Analytic Moufang loops in the large, Alg. Logic \textbf{18}, 325-347 (1980)

\bibitem{K2}
Kupershmidt, B. A.: What a Classical $r$-Matrix Really Is, J. Nonlin. Math. Phys., \textbf{6}(4), 448-488 (1999) % DOI: 10.2991/jnmp.1999.6.4.5

\bibitem{kuzmin68}
Kuzmin, E. N.: Malcev algebras and their representations, Alg. Logic, \textbf{7},
233-244 (1968)

\bibitem{kuzmin71:conectMalcevanMoufangloops}
Kuzmin, E. N.: The connection between Malcev algebras and analytic Moufang loops, Alg. Logic,
\textbf{10}, 1-14 (1971)

\bibitem{LarssonSilvJA2005:QuasiHomLieCentExt2cocyid}
Larsson, D., Silvestrov, S. D.: Quasi-Hom-Lie algebras, central extensions and $2$-cocycle-like identities, J. Algebra, \textbf{288}, 321-344 (2005) (Preprints in Mathematical Sciences 2004:3, LUTFMA-5038-2004, Centre for Mathematical Sciences, Department of Mathematics, Lund Institute of Technology, Lund University (2004))

%\bibitem{LarssonSilv:quasiLiealg}
%Larsson, D., Silvestrov, S. D.: Quasi-Lie algebras, In:
%Fuchs, J., Mickelsson, J., Rozenblioum, G., Stolin, A., Westerberg, A. (Eds), "Noncommutative Geometry and Representation Theory in Mathematical Physics", %Contemp. Math., 391, Amer. Math. Soc., Providence, RI, 241-248 (2005)
%(Preprints in Mathematical Sciences 2004:30, LUTFMA-5049-2004, Centre for Mathematical Sciences, Department of Mathematics, Lund Institute of Technology, %Lund University (2004))

%\bibitem{LSGradedquasiLiealg}
%Larsson, D., Silvestrov, S. D.: Graded quasi-Lie agebras, Czechoslovak J. Phys. \textbf{55}, 1473-1478 (2005)

%\bibitem{LarssonSilv:quasidefsl2}
%Larsson, D., Silvestrov, S. D.: Quasi-deformations of $sl_2(\mathbb{F})$ using twisted derivations, Comm. Algebra, \textbf{35}, 4303-4318 (2007)
%(Preprint in Mathematical Sciences 2004:26, LUTFMA-5047-2004, Centre for Mathematical Sciences, Lund Institute of Technology, Lund University (2004). %arXiv:math/0506172 [math.RA] (2005))

%\bibitem{LarssonSigSilvJGLTA2008:QuasiLiedefFttN}
%Larsson, D., Sigurdsson, G., Silvestrov, S. D.: Quasi-Lie deformations on the algebra $\mathbb{F}[t]/(t^N)$, J. Gen. Lie Theory Appl. \textbf{2}(3), %201-205 (2008)

%\bibitem{LarssonSilvestrovGLTMPBSpr2009:GenNComplTwistDer}
%Larsson, D., Silvestrov, S. D.: On generalized $N$-complexes comming from twisted derivations, In: Silvestrov, S., Paal, E., Abramov, V., Stolin, A. (Eds.),
%Generalized Lie Theory in Mathematics, Physics and Beyond, Springer-Verlag, Ch. 7, 81-88 (2009)

\bibitem{Loday01}
Loday, J-L.: Dialgebras, In: Loday, J.-L., Frabetti, A., Chapoton, F., Goichot, F. (Eds), Dialgebras and Related Operads, Lecture Notes in Math. \textbf{1763}, Springer, Berlin, 7-66 (2001). (Pr\'epublication de l'Inst de Recherche Math. Avanc\'ee (Strasbourg), \textbf{14}, 61 pp (1999))


\bibitem{LodayRonco04}
Loday, J.-L., Ronco, M.:
Trialgebras and families of polytopes.
Contemp. Math. \textbf{346}, 369-398 (2004)


%\bibitem{MaMakhSil:CurvedOoperatorSyst}
%Ma, T., Makhlouf, A., Silvestrov, S.:
%Curved $\mathcal{O}$-operator systems, 17pp, arXiv: 1710.05232 [math.RA] (2017)

%\bibitem{MaMakhSil:RotaBaxbisyscovbialg}
%Ma, T., Makhlouf, A., Silvestrov, S.:
%Rota-Baxter bisystems and covariant bialgebras, 30 pp, arXiv:1710.05161[math.RA] (2017)

%\bibitem{MaMakhSil:RotaBaxCosyCoquasitriMixBial}
%Ma, T., Makhlouf, A., Silvestrov, S.:
%Rota–Baxter cosystems and coquasitriangular mixed bialgebras, J. Algebra Appl. \textbf{20}(04), 2150064 (2021)

%\bibitem{MaZheng:RotaBaxtMonoidalHomAlg}
%Ma, T., Zheng, H.: Some results on Rota-Baxter monoidal Hom-algebras, Results Math. \textbf{72} (1-2), 145-170 (2017)
%doi.org/10.1007/s00025-016-0641-9

\bibitem{MaZheng}
Ma, T.,  Zheng, H.: $(m, n)$-Hom-Lie algebras, Publ. Math. Debrecen, \textbf{92}, 59-78 (2018)

\bibitem{Madariaga}
Madariaga, S.: Splitting of operations for alternative and Malcev structures,
Communications in Algebra, \textbf{45}(1), 183-197 (2014)

%\bibitem{MabroukNcibSilvestrov2020:GenDerRotaBaxterOpsnaryHomNambuSuperalgs}
%Mabrouk, S., Ncib, O., Silvestrov, S.: Generalized Derivations and Rota-Baxter Operators of $n$-ary Hom-Nambu Superalgebras, Adv. Appl. Clifford Algebras, \textbf{31}(32), (2021). (arXiv:2003.01080[math.QA]) % https://doi.org/10.1007/s00006-020-01115-2

\bibitem{MakhloufHomdemdoformRotaBaxterHomalg2011}
Makhlouf, A.: Hom-dendriform algebras and Rota-Baxter Hom-algebras, In: Bai, C., Guo, L., Loday, J.-L. (eds.), Nankai Ser. Pure Appl. Math. Theoret. Phys., \textbf{9}, World Sci. Publ. 147-171 (2012)
% doi:10.1142/9789814365123_0008

%\bibitem{Makhl:HomaltHomJord}
%Makhlouf, A.: Hom-alternative algebras and Hom-Jordan algebras, Int. Elect. Journ. Alg., \textbf{8}, 177-190 (2010). (arXiv:0909.0326 (2009))

%\bibitem{Makhlouf2010:ParadigmnonassHomalgHomsuper}
%Makhlouf, A.: Paradigm of nonassociative Hom-algebras and Hom-superalgebras, In:
%Proceedings of Jordan Structures in Algebra and Analysis Meeting, 145-177 (2010)
%(arXiv:1001.4240v1)

\bibitem{ms:homstructure}
Makhlouf, A., Silvestrov, S. D.:
Hom-algebra structures. J. Gen. Lie Theory Appl. \textbf{2}(2), 51--64 (2008)
(Preprints in Mathematical Sciences  2006:10, LUTFMA-5074-2006, Centre for Mathematical Sciences, Department of Mathematics, Lund Institute of Technology, Lund University (2006))

%\bibitem{MakhSil:HomHopf}
%Makhlouf, A., Silvestrov, S.:
%Hom-Lie admissible Hom-coalgebras and Hom-Hopf algebras,
%In: Silvestrov, S., Paal, E., Abramov, V., Stolin, A. (Eds.),
%Generalized Lie Theory in Mathematics, Physics and Beyond, Springer-Verlag, Berlin, Heidelberg, Ch. 17, 189-206 (2009) (Preprints in Mathematical Sciences, Lund University, Centre for Mathematical Sciences, Centrum Scientiarum Mathematicarum (2007:25) LUTFMA-5091-2007 and in arXiv:0709.2413 [math.RA] (2007))

%\bibitem{MakhSilv:HomAlgHomCoalg}
%Makhlouf, A., Silvestrov, S. D.:
%Hom-algebras and Hom-coalgebras, J. Algebra Appl. \textbf{9}(4), 553-589 (2010) (Preprints in Mathematical Sciences, Lund University, Centre for Mathematical Sciences, Centrum Scientiarum Mathematicarum, (2008:19) LUTFMA-5103-2008. arXiv:0811.0400
%[math.RA] (2008)) % DOI: 10.1142/S0219498810004117

\bibitem{MakhSilv:HomDeform}
Makhlouf, A., Silvestrov, S.: Notes on $1$-parameter formal deformations of Hom-associative and Hom-Lie algebras, Forum Math. \textbf{22}(4), 715-739 (2010)
(Preprints in Mathematical Sciences, Lund University, Centre for Mathematical Sciences, Centrum Scientiarum Mathematicarum, (2007:31) LUTFMA-5095-2007. arXiv:0712.3130v1 [math.RA] (2007))

%\bibitem{MakYau:RotaBaxterHomLieadmis}
%Makhlouf, A., Yau, D.: Rota-Baxter Hom-Lie admissible algebras, Comm. Alg., \textbf{23}(3), 1231-1257 (2014)

\bibitem{Malcev55:anltcloops}
Malcev, A. I.: Analytic loops, Mat. Sb. \textbf{36}(3), 569-576 (1955)

\bibitem{myung}
Myung, H. C.: Malcev-admissible algebras, Progress in Math. \textbf{64}, Birkh\"{a}user, (1986)

\bibitem{nagy}
Nagy, P. T.: Moufang loops and Malcev algebras, Sem. Sophus Lie, \textbf{3}, 65-68 (1993)

\bibitem{okubo}
Okubo, S.: Introduction to octonion and other non-associative algebras in physics, Cambridge Univ. Press, (1995)

%\bibitem{RichardSilvestrovJA2008}
%Richard, L., Silvestrov, S. D.: Quasi-Lie structure of $\sigma$-derivations of $\mathbb{C}[t^{\pm1}]$,
%J. Algebra \textbf{319}(3), 1285-1304 (2008) (arXiv:math/0608196[math.QA] (2006). Preprints in mathematical sciences (2006:12), LUTFMA-5076-2006, Centre for Mathematical Sciences, Lund University (2006))

%\bibitem{RichardSilvestrovGLTbnd2009}
%Richard, L., Silvestrov, S. D.:
%A note on quasi-Lie and Hom-Lie structures of $\sigma$-derivations of
%${\mathbb C}[z_1^{\pm 1},\ldots,z_n^{\pm 1}]$, In: Silvestrov, S., Paal, E., Abramov, V., Stolin, A. (Eds.), Generalized Lie Theory in Mathematics, Physics and Beyond, Springer-Verlag, Ch. 22, 257-262, (2009)

\bibitem{Rota}
Rota, G.-C.: Baxter algebras and combinatorial identities I., Bull. Amer. Math. Soc. \textbf{75}, 325-329 (1969)

%\bibitem{SaadaouSilvestrov:lmgderivationsBiHomLiealgebras}
%Saadaou, N, Silvestrov, S.: On $(\lambda,\mu,\gamma)$-derivations of BiHom-Lie algebras,	arXiv:2010.09148 [math.RA], (2020)

\bibitem{sabinin}
Sabinin, L. V.: Smooth quasigroups and loops, Kluwer Academic, (1999)

%\bibitem{Sagle}
%Sagle, A. A.: Malcev algebras, Trans. Amer. Math. Soc. %\textbf{101}, 426-458 (1961)

\bibitem{Sheng:homrep}
Sheng, Y.: Representations of Hom-Lie algebras, Algebr. Reprensent. Theory \textbf{15}, 1081-1098 (2012)

%\bibitem{ShengBai:homLiebialg}
%Sheng, Y.,  Bai,  C.: A  new  approach  to  Hom-Lie  bialgebras,  J. Algebra,  \textbf{399},  232-250  (2014)

%\bibitem{SigSilv:CzechJP2006:GradedquasiLiealgWitt}
%Sigurdsson, G., Silvestrov, S.: Graded quasi-Lie algebras of Witt type, Czechoslovak J. Phys. \textbf{56}, 1287-1291 (2006)

%\bibitem{SigSilv:GLTbdSpringer2009}
%Sigurdsson, G., Silvestrov, S.: Lie color and Hom-Lie algebras of Witt type and their central extensions, In: Silvestrov, S., Paal, E., Abramov, V., %Stolin, A. (Eds.), Generalized Lie Theory in Mathematics, Physics and Beyond, Springer-Verlag, Berlin, Heidelberg, Ch. 21, 247-255 (2009)

%\bibitem{SilvestrovParadigmQLieQhomLie2007}
%Silvestrov, S.: Paradigm of quasi-Lie and quasi-Hom-Lie algebras and quasi-defor\-mations, In "New techniques in Hopf algebras and graded ring theory", K. %Vlaam. Acad. Belgie Wet. Kunsten (KVAB), Brussels, 165-177 (2007)

%\bibitem{SilvestrovZardeh2021:HNNextinvolmultHomLiealg}
%Silvestrov, S., Zargeh, C.: HNN-extension of involutive multiplicative Hom-Lie algebras, arXiv:2101.01319 [math.RA] (2021).

\bibitem{sheng}
 Sheng, Y.,  Chen, D.: Hom-Lie 2-algebras, J. Algebra, \textbf{376}, 174-195 (2013)

\bibitem{Songtang}
Song, L., Tang, R.: Derivation Hom-Lie 2-algebras and non-abelian extensions of regular
Hom-Lie algebras, J. Algebra Appl. \textbf{17}, 1850081 (2018)

\bibitem{Q.Sun}
Sun, Q.: On Hom-Prealternative Bialgebras, Algebr. Represent. Theor. \textbf{19}, 657-677 (2016)

\bibitem{Yau:HomYangBaHomLiequasitribial}
Yau  D.:  The  Hom-Yang-Baxter  equation,  Hom-Lie  algebras  and  quasi-triangular  bialgebras,  J. Phys. A.  \textbf{42}, 165--202  (2006)

%\bibitem{Yau:ModuleHomalg}
%Yau, D.: Module Hom-algebras, arXiv:0812.4695[math.RA] (2008)

\bibitem{Yau:HomEnv}
Yau, D.: Enveloping algebras of Hom-Lie algebras, J. Gen. Lie Theory Appl. \textbf{2}(2), 95-108 (2008). (arXiv:0709.0849 [math.RA] (2007))

\bibitem{Yau:HomHom}
Yau, D.: Hom-algebras and homology, J. Lie Theory \textbf{19}(2), 409-421 (2009)

%\bibitem{Yau:HombialgcomoduleHomalg}
%Yau, D.: Hom-bialgebras and comodule Hom-algebras, \textit{Int. Electron.~J. Algebra} \textbf{8}, 45-64  (2010). (arXiv:0810.4866[math.RA] (2008))

\bibitem{Yau}
Yau, D.: Hom-Malcev, Hom-alternative, and Hom-Jordan algebras, Int. Electron. J. Algebra, \textbf{11},  177-217 (2012)

\bibitem{zhengguo}
Zheng, S., Guo, L.: Free involutive  Hom-semigroups and Hom-associative algebras, Front. Math. China \textbf{11}, 497-508 (2016)

\end{thebibliography}
\end{document}